\documentclass[11pt,reqno,tbtags,a4paper]{amsart}
\usepackage{amssymb}
\usepackage{url}
\usepackage[square,numbers]{natbib}
\bibpunct[, ]{[}{]}{;}{n}{,}{,}

\title%[]
{The hiring problem with rank-based strategies}

\date{11 September, 2018}
\author{Svante Janson}
\thanks{Partly supported by the Knut and Alice Wallenberg Foundation}
\address{Department of Mathematics, Uppsala University, PO Box 480,
SE-751~06 Uppsala, Sweden}
\email{svante.janson@math.uu.se}
\urladdr{http://www.math.uu.se/svante-janson}
\numberwithin{equation}{section}

\renewcommand\le{\leqslant}
\renewcommand\ge{\geqslant}

\allowdisplaybreaks

\theoremstyle{plain}% default
\newtheorem{theorem}{Theorem}[section]
\newtheorem{lemma}[theorem]{Lemma}

\newtheorem{corollary}[theorem]{Corollary}

\theoremstyle{definition}
\newtheorem{example}[theorem]{Example}

\newtheorem{problem}[theorem]{Problem}
\newtheorem{remark}[theorem]{Remark}

\newtheorem*{ack}{Acknowledgement}

\theoremstyle{remark}
\newtheorem*{claim}{Claim}

\newenvironment{romenumerate}[1][-10pt]{% optional argument changes indentation
\addtolength{\leftmargini}{#1}\begin{enumerate}% gives (i), (ii) etc.
 \renewcommand{\labelenumi}{\textup{(\roman{enumi})}}%
 }{\end{enumerate}}

\newcounter{oldenumi}
{\setcounter{oldenumi}{\value{enumi}}
\begin{romenumerate} \setcounter{enumi}{\value{oldenumi}}}
{\end{romenumerate}}

\newcounter{thmenumerate}

\newcounter{xenumerate}   %no left indentation; thus wider lines

\newcommand\pfitemx[1]{\par#1:}
\newcommand\pfitemref[1]{\pfitemx{\ref{#1}}}

\newcommand{\refT}[1]{Theorem~\ref{#1}}

\newcommand{\refC}[1]{Corollary~\ref{#1}}

\newcommand{\refL}[1]{Lemma~\ref{#1}}

\newcommand{\refR}[1]{Remark~\ref{#1}}

\newcommand{\refS}[1]{Section~\ref{#1}}
\newcommand{\refSs}[1]{Sections~\ref{#1}}

\newcommand{\refE}[1]{Example~\ref{#1}}
\newcommand{\refEs}[1]{Examples~\ref{#1}}

\newcommand{\refApp}[1]{Appendix~\ref{#1}}
\newcommand{\refTab}[1]{Table~\ref{#1}}

\newcommand\REM[1]{{\raggedright\texttt{[#1]}\par\marginal{XXX}}}
\newcommand\XREM[1]{\relax}

\begingroup
  \divide\count255 by 60
  \multiply\count255 by -60
  \advance\count255 by \time
  \ifnum \count255 < 10 \xdef\klockan{\the\count1.0\the\count255}
  \else\xdef\klockan{\the\count1.\the\count255}\fi
\endgroup

\newcommand{\sumko}{\sum_{k=0}^\infty}
\newcommand{\summo}{\sum_{m=0}^\infty}

\newcommand{\sumj}{\sum_{j=1}^\infty}
\newcommand{\sumk}{\sum_{k=1}^\infty}
\newcommand{\summ}{\sum_{m=1}^\infty}

\newcommand{\sumkn}{\sum_{k=1}^n}
\newcommand{\sumkm}{\sum_{k=1}^m}

\newcommand{\sumkom}{\sum_{k=0}^m}
\newcommand{\sumkomi}{\sum_{k=0}^{m-1}}
\newcommand{\prodiq}{\prod_{i=1}^q}
\newcommand{\prodk}{\prod_{k=1}^\infty}

\newcommand{\prodoj}{\prod_{j=0}^\infty}

\newcommand\set[1]{\ensuremath{\{#1\}}}

\newcommand\xpar[1]{(#1)}
\newcommand\bigpar[1]{\bigl(#1\bigr)}
\newcommand\Bigpar[1]{\Bigl(#1\Bigr)}
\newcommand\biggpar[1]{\biggl(#1\biggr)}
\newcommand\lrpar[1]{\left(#1\right)}
\newcommand\bigsqpar[1]{\bigl[#1\bigr]}

\newcommand\xcpar[1]{\{#1\}}
\newcommand\bigcpar[1]{\bigl\{#1\bigr\}}

\newcommand\bigabs[1]{\bigl|#1\bigr|}
\newcommand\Bigabs[1]{\Bigl|#1\Bigr|}

\newcommand\lrabs[1]{\left|#1\right|}
\def\rompar(#1){\textup(#1\textup)}    % usage: \rompar(...)
\newcommand\xfrac[2]{#1/#2}

\newcommand\parfrac[2]{\lrpar{\frac{#1}{#2}}}
\newcommand\bigparfrac[2]{\bigpar{\frac{#1}{#2}}}
\newcommand\Bigparfrac[2]{\Bigpar{\frac{#1}{#2}}}

\def\xexp(#1){e^{#1}}
\newcommand\ceil[1]{\lceil#1\rceil}
\newcommand\floor[1]{\lfloor#1\rfloor}

\newcommand\ntoo{\ensuremath{{n\to\infty}}}

\newcommand\ktoo{\ensuremath{{k\to\infty}}}
\newcommand\mtoo{\ensuremath{{m\to\infty}}}

\newcommand\itoo{\ensuremath{{i\to\infty}}}

\newcommand\norm[1]{\|#1\|}
\newcommand\bignorm[1]{\bigl\|#1\bigr\|}

\newcommand\punkt{.\spacefactor=1000}    % om problem!
\newcommand\iid{i.i.d\punkt}    
\newcommand\ie{i.e\punkt}
\newcommand\eg{e.g\punkt}
\newcommand\viz{viz\punkt}
\newcommand\cf{cf\punkt}
\newcommand{\as}{a.s\punkt}

\newcommand\whp{w.h.p\punkt}

\newcommand{\tend}{\longrightarrow}
\newcommand\dto{\overset{\mathrm{d}}{\tend}}
\newcommand\pto{\overset{\mathrm{p}}{\tend}}
\newcommand\asto{\overset{\mathrm{a.s.}}{\tend}}
\newcommand\eqd{\overset{\mathrm{d}}{=}}

\newcommand\op{o_{\mathrm p}}

\newcommand\Op{O_{\mathrm p}}

\newcommand\bbR{\mathbb R}

\newcommand\bbN{\mathbb N}

\newcommand\bbZ{\mathbb Z}

\newcounter{CC}
\newcounter{cc}
\renewcommand\Re{\operatorname{Re}}

\newcommand\E{\operatorname{\mathbb E{}}}
\renewcommand\P{\operatorname{\mathbb P{}}}

\newcommand\Var{\operatorname{Var}}

\newcommand\Exp{\operatorname{Exp}}

\newcommand\Bi{\operatorname{Bi}}
\newcommand\Bin{\operatorname{Bin}}
\newcommand\Be{\operatorname{Be}}
\newcommand\Ge{\operatorname{Ge}}

\newcommand\ga{\alpha}
\newcommand\gb{\beta}
\newcommand\gd{\delta}

\newcommand\gam{\gamma}
\newcommand\gG{\Gamma}

\newcommand\kk{\kappa}
\newcommand\gl{\lambda}

\newcommand\gs{\sigma}

\newcommand\gss{\sigma^2}

\newcommand\eps{\varepsilon}

\renewcommand\phi{\xxx}  %% WARNING

\newcommand\cH{\mathcal H}

\newcommand\cJ{\mathcal J}

\newcommand\cL{{\mathcal L}}

\newcommand\cR{{\mathcal R}}

\newcommand\tT{{\widetilde T}}

\newcommand\tW{\widetilde W}

\newcommand\tY{{\tilde Y}}
\newcommand\tZ{{\tilde Z}}

\newcommand\indic[1]{\boldsymbol1\xcpar{#1}} 
\newcommand\bigindic[1]{\boldsymbol1\bigcpar{#1}}

\newcommand\qw{^{-1}}
\newcommand\qww{^{-2}}
\newcommand\qq{^{1/2}}

\newcommand\intoo{\int_0^\infty}

\newcommand\ooi{(0,1]}
\newcommand\ooo{[0,\infty)}

\newcommand\dtv{d_{\mathrm{TV}}}

\newcommand\dd{\,\mathrm{d}}

\newcommand{\mgf}{moment generating function}
\newcommand{\chf}{characteristic function}

\newcommand\rv{random variable}
\newcommand\lhs{left-hand side}
\newcommand\rhs{right-hand side}

\newcommand\etto{\bigpar{1+o(1)}}

\newcommand\xoo{_1^\infty}
\newcommand\oxoo{_0^\infty}

\newcommand\XX{X^*}
\newcommand\XXo[1]{\XX_{(#1)}}
\newcommand\hV{\widehat V}
\newcommand\xsim{\in}
\newcommand\KK{L}
\newcommand\hgs{\hat\sigma}
\newcommand\hgss{\hgs^2}
\newcommand\YYY{(Y_k)\xoo}
\newcommand\hmedian{`hiring above the median'}

\newcommand\hpercentile[1]{`the $#1$-percentile rule'}
\newcommand\hbest[1]{`hiring above the $#1$-th best'}
\newcommand\hmean{`hiring above the mean'}
\newcommand\Wmed{W_{\mathrm{med}}}
\newcommand\nuq{\nu}
\newcommand\xr{\tilde r}
\newcommand\xrho{\tilde \rho}
\newcommand\xW{\tW}
\newcommand\xZ{\tZ}
\newcommand\xY{\tY}
\newcommand\hy{\hat y}
\newcommand\xm{\widetilde m}
\newcommand\Hz{\cH(z)}
\newcommand\Mle[1]{M^{\le #1}}
\newcommand\Mgt[1]{M^{> #1}}
\newcommand\YYn{Y^*_n}
\newcommand\YYx[1]{Y^*_{#1}}
\newcommand\bX{\bar X}
\newcommand\chr{\check r}
\newcommand\chN{\check N}
\newcommand\chW{\check W}
\newcommand\cJx{\cJ^*}
\newcommand\rx{r_*}
\newcommand\LL{\widehat L}
\newcommand\xtau{\nu}
\newcommand\CS{Cauchy--Schwarz}
\newcommand\CSineq{\CS{} inequality}
\newcommand{\Levy}{L\'evy}

\begin{document}

\begin{abstract} 
The hiring problem is studied for general strategies based only on the
relative ranking of the candidates; this includes some well known 
strategies studied before such as hiring above the median. 
We give general limit theorems
for the number of hired candidates and some other properties, extending
previous results. 
The results exhibit a dichotomy between two classes of rank-based
strategies: either the asymptotics of the process are determined by the
early events, with a.s.\ convergence of suitably normalized random
variables, or there is a mixing behaviour without long-term memory and with
asymptotic normality.
\end{abstract}

\maketitle

\section{Introduction}\label{S:intro}

The \emph{hiring problem} is a variant of the well-known secretary problem,
in which we want to hire many good applicants and not just the best.
An informal formulation is that a large 
number of candidates are examined (interviewed) one by one; immediately
after each interview we decide whether to hire the candidate or not, based
on the value of the candidate (which is assumed to be revealed during the
interview) and of the values of the candidates seen earlier.
This is thus an on-line type of decision problem.
The mathematical model assumes that the values of the candidates are \iid{}
\rv{s}, with some continuous distribution (which prevents ties).
See below and \refS{Sgeneral} for formal details.

There are two conflicting aims in the hiring problem: 
we want to hire (rather) many
candidates but we also want them to be good. 
Thus there is no single goal in the hiring problem, and thus we cannot
talk about an optimal solution. 
Instead, the mathematical problem is to analyse
properties of various proposed strategies. 
The property that has been most studied is
the number of accepted candidates among the first $n$ examined, here denoted
$M_n$. We will also study the inverse function $N_m$, the number of
candidates examined until $m$ are accepted.
Some other properties, such as the distribution of the value of the
accepted candidates,
are discussed in \refSs{Scond}--\ref{Saccepted}.

The hiring problem seems to have been studied first by \citet{Preater},
and later (independently) by \citet{BroderEtAl}, who also
invented the name `hiring problem'.
Further papers studying the hiring problem under various strategies are
\citet{KriegerPS-rank2007,KriegerPS-beat2008,KriegerPS-extreme2010},
\citet{ArchibaldMartinez2009},
\citet{GaitherWard2012},
\citet{HelmiP-ANALCO2012,HelmiP-median2013},
\citet{HelmiMP-Latin2012,HelmiMP-best2014}; 
see \cite{HelmiMP-best2014} for a more detailed history.

There are two main groups of strategies. 
In the present paper we study only \emph{rank-based strategies}, \ie,
strategies that depend only on the rank of
the candidate among the ones seen so far; in other words, on the relative
order of the values of the candidates.
A typical example is \hmedian, see below; see also
\cite{KriegerPS-rank2007,
BroderEtAl,
ArchibaldMartinez2009,
GaitherWard2012,
HelmiP-ANALCO2012,
HelmiMP-Latin2012,
HelmiP-median2013,
HelmiMP-best2014}.
In statistical terms, the values of the candidates are regarded as on an
ordinal scale.
Thus, the distribution of the value of a candidate
does not matter, and can freely be chosen as \eg{} uniform (an obvious
standard choice used by some previous authors) or exponential (used in
the analysis in the present paper). 
Furthermore, 
for rank-based strategies, it is equivalent 
(for a fixed $n$) to assume that the values of the
first $n$ candidates form a uniformly random permutation of \set{1,\dots,n}
\cite{KriegerPS-rank2007,
ArchibaldMartinez2009,
GaitherWard2012,
HelmiP-ANALCO2012,
HelmiMP-Latin2012,
HelmiP-median2013,
HelmiMP-best2014}.

The alternative is to use a strategy depending on the actual values; a
typical example is \hmean{}
\cite{Preater,
BroderEtAl,
KriegerPS-beat2008,KriegerPS-extreme2010}.
For such strategies, the results  depend on the given distribution of the
value of a candidate; several different distributions have been investigated
in the papers just mention. Such strategies will not be considered in the
present paper.

In the present paper we thus study rank-based strategies.
More precisely, following \citet{KriegerPS-rank2007},
we consider strategies of the following form (which seems to include all
reasonable rank-based strategies).
We assume throughout the paper that we are given a sequence of integers
$r(m)$, $m\ge0$, 
such that %$r(0)=1$ and
\begin{equation}
  \label{eq:rr}
r(0)=1\qquad\text{and}\qquad
  r(m)\le r(m+1)\le r(m)+1,\quad m\ge0.
\end{equation}
Note that this implies $1\le r(m)\le m+1$ for every $m\ge0$.
Then the strategy is that if so far $m\ge0$ candidates have been accepted, then 
(if $r(m)\le m$)
the next candidate is accepted if her value is above the $r(m)$-th best
value of the ones already accepted. 
If $r(m)=m+1$, we always accept the next candidate.

\begin{remark}
   Some authors have chosen to define smaller values as better.
This is obviously equivalent, but some care has to be taken when comparing
definitions and results. (In the discussion here, we have when
necessary translated to our setting.)
\end{remark}

One basic example is \emph{\hmedian}, where we hire every
candidate that is better than the median of the ones seen so far. For $m$
odd, this means taking $r(m)=(m+1)/2$, but for $m$ even, there are two
possible values of the median.
Tradition established by previous authors 
\cite{BroderEtAl,
ArchibaldMartinez2009,
%HelmiP-ANALCO2012,
HelmiP-median2013}
says that we choose the smaller 
value as the threshold. (This is thus the less restrictive policy of the two
possibilities.) This means taking $r(m)=m/2+1$ when $m$ is even, so we can
summarize the strategy \hmedian{} by 
\begin{equation}\label{hmedian}
  r(m):=\floor{m/2}+1 = \ceil{(m+1)/2}, \qquad m\ge0.
\end{equation}
The sequence (starting with $r(0)$) thus is $1,1,2,2,3,3\dots$

The alternative interpretation of `median' when $m$ is even is instead a
special case of the strategy known as \emph{\hpercentile{\ga}}, 
which, again by tradition 
\cite{KriegerPS-rank2007,
%GaitherWard2012,
HelmiP-median2013}, 
is defined by
\begin{equation}\label{hpercentile}
  r(m):=\ceil{\ga m}, \qquad m\ge1.
\end{equation}
In the case $\ga=\frac12$, we thus take $r(m)=m/2$ when $m\ge2$ is even,
meaning a smaller $r(m)$ and thus a larger threshold than in \hmedian;
the sequence (starting with $r(0)$) is $1,1,1,2,2,3,\dots$

%An alternative is \hquantile{(1-\ga)} \cite{ArchibaldMartinez2009}
% with $\ceil{\ga(m+1)}. Incorrect?

A third simple example is \emph{\hbest{r}}, where $r\ge1$ is a fixed
number
\cite{ArchibaldMartinez2009,
HelmiMP-Latin2012,
HelmiMP-best2014}.
This means $r(m)=r$ when $m\ge r$; we always hire the $r$ first
candidates, so the complete definition in accordance with \eqref{eq:rr} is
\begin{equation}\label{hbest}
  r(m) := \min\set{r,m+1}, 
\qquad m\ge0.
\end{equation}
Note that the case $r=1$ gives the strategy of hiring only the candidates
that are better than everyone seen earlier, i.e., the records. 

The present paper gives a general analysis of strategies of the type above,
with an arbitrary sequence $r(m)$ satisfying \eqref{eq:rr}.
Our main results give the asymptotic distribution of $N_m$ (in general) 
and $M_n$ (under weak regularity assumptions on $r(m)$), see in particular
Theorems \ref{TN4}, \ref{TNsmall}, \ref{TI1} and \ref{TI2}.
In particular, this gives new proofs of known results (and some new)
for the examples above.

It turns out that there is a dichotomy:
the general results are of two different types, depending on whether
$\sum_m r(m)\qww <\infty$ or 
$\sum_m r(m)\qww =\infty$;
we will call these two cases \emph{large $r(m)$} and \emph{small $r(m)$},
  respectively.
Note that the strategies 
\hmedian{} and \hpercentile{\ga}
have large $r(m)$, while
\hbest{r} has small $r(m)$; 
indeed, the limit theorems found by previous
researchers for these cases are of different types, compare \eg{}
\cite{HelmiP-median2013} and \cite{HelmiMP-best2014}.

The main differences between the two cases can be summarized as follows,
at least assuming some further regularity of $r(m)$.
For simplicity, we consider here only $M_n$;  the same types of results hold
for $N_m$.

{\setlength{\leftmargini}{2\parindent}
\begin{description}
\setlength{\itemindent}{-\parindent}
\item[large $r(m)$, $\sum_m r(m)\qww <\infty$] 
$M_n/\E M_n$ converges to a non-degenerate distribution on $\ooo$.
(Thus, the limit is not normal.)
Furthermore, $M_n/\E M_n$ converges a.s. Hence, the limit and the asymptotic
behaviour are essentially determined by what happens very early, i.e., by the
values of the first few candidates.
This also means a strong long-range dependence in the sequence $M_n$.

\item[small $r(m)$, $\sum_m r(m)\qww =\infty$]  
Asymptotic normality of $M_n$. 
There is no long-range dependency; 
instead $M_n$ is asymptotically independent of what happened with the first
$n_0$ candidates, for any fixed $n_0$.
In particular, there is no \as{} convergence.
\end{description}}

Intuitively, the reason for the difference between the two cases is that
when $r(m)$ is small,
each accepted candidate
has (typically) a rather large influence 
on the threshold, and thus on the future of the process, and these influences
add up and eventually dominate over the influences of the values of the
first candidates, 
while if $r(m)$ is large, then the influences of later candidates are small,
and the effects of the first few candidates dominate.

We state here two theorems that exemplify our main results, 
one for large $r(m)$ and one for small.
In both cases, we assume a regularity condition on $r(m)$ yielding simpler
results; proofs and more general results are given in
\refSs{Slarge}--\ref{Slinear}  and \ref{Ssmall}. 
See also the examples in \refS{Sex}, including the counterexample \refE{Eirreg}.

First, consider the case $r(m)=\ga m+ O(1)$, where $0<\ga\le1$; note that
this includes \hmedian{} (with $\ga=\frac12$)
and \hpercentile{\ga}, and that such $r(m)$ are large.

\begin{theorem}
  \label{TI1}
Suppose that $r(m)=\ga m+ O(1)$, where $0<\ga\le1$. 
Then, 
\begin{equation}\label{limW}
  M_n / n^\ga \asto W,
\end{equation}
for a random variable $W$, which can be represented as in \eqref{W} below.
Furthermore, all moments converge in \eqref{limW}, and, for every $s\ge0$,
\begin{align}\label{ti1}
  \E M_n^s/n^{\ga s} \to \E W^s 
&= \frac{\gG(s+1)}{\gG(s\ga+2)} \prodk \frac{1+s/k}{1+s\ga /r(k)}
%\\
%&= \frac{\gG(s+1)}{\gG(s\ga+2)} \prodk \frac{r(k)+sr(k)/k}{r(k)+s\ga}
<\infty.
\end{align}
\end{theorem}

The moments in \eqref{ti1} can be explicitly computed in many cases, see
\cite{GaitherWard2012},
\refT{Tper} and Examples \ref{Emedian}--\ref{Epercentile}.

For our example result in the case of small $r(m)$, we assume that the
sequence
$r(m)$ is
regularly varying. 
(See e.g.\ \cite[p.~52]{RegVar} for definition,
and \cite{mixing} for definition of a mixing limit theorem.)

\begin{theorem}
  \label{TI2}
Assume that $r(m)$ is a regularly varying sequence such that $\sum_m
r(m)\qww=\infty$.
Let $\mu(n)$, $n\ge1$, be a sequence of real numbers such that
\begin{equation}\label{ti2mu}
  \sum_{k=1}^{\mu(n)}\frac{1}{r(k)} =\log n + O(1),
\end{equation}
and let
\begin{align}\label{ti2gb}
  \gb(m)^2&:=r(m)^2\sumkm \frac{1}{r(k)^2}
=\sumkm \frac{r(m)^2}{r(k)^2},
\\
\gam(n)&:=\gb\bigpar{\floor{\mu(n)}}.\label{ti2gam}
\end{align}
Then, as \ntoo,
\begin{equation}
  \label{ti2}
\frac{M_n-\mu(n)}{\gam(n)}\dto N(0,1).
\end{equation}
Furthermore, \eqref{ti2} is mixing.
\end{theorem}

\refS{Sgeneral} contain some basic general results. Then
\refT{TI1} and related results for large $r(m)$ are  proved in
\refSs{Slarge}--\ref{Slarge++}, while \refT{TI2} and related results for
small $r(m)$ are  proved in \refS{Ssmall}. \refS{Sex} contains some
examples.
The remaining sections consider some additional properties that have been
considered by previous authors.
\refS{Scond} consider results conditioned on the value of the first candidate.
\refS{Sgaps} treat the probability of accepting the next candidate, and also
the number of unsuccessful candidates since the last accepted one.
\refS{Saccepted} studies the distribution of the accepted values.

\begin{remark}
Some previous papers, in particular
\cite{HelmiP-median2013,HelmiMP-best2014},
contain also interesting exact results on the exact distribution of $M_n$
for finite $n$. We do not consider such results here.
\end{remark}

\begin{ack}
This research was begun during a lecture by Conrado Mart\'inez at the 
conference celebrating the 60th birthday of Philippe Flajolet in December
2008;
Conrado talked about the hiring problem, and I got the basic idea of the method
used here, and made some notes. It then took me almost 10 years to return to
the problem and develop the ideas in the notes, with further inspiration
from papers by 
Conrado and others that were written in the meantime.
I thank Conrado Mart\'inez for an inspiring lecture, and I find it fitting
to dedicate this paper to the memory of Philippe Flajolet.
\end{ack}

\section{Notation}
 $\Exp(a)$ denotes the exponential distribution with \emph{rate} $a$.
In other words, if $X\xsim\Exp(a)$, then $\P(X>x)=e^{-ax}$ for $x\ge0$;
equivalently, $aX\xsim\Exp(1)$. Hence, $\E X=1/a$.

$\Ge(p)$ denotes the Geometric distribution started at 1 (also called First
Success distribution), with
$\P(X=n)=p(1-p)^{n-1}$, $n\ge1$.

$E$, $E_i$, $E'_j$ and so on  
will always denote $\Exp(1)$ variables, independent of each
other.

$\gam$ is Euler's constant.

$\asto$, $\pto$ and $\dto$ denote convergence almost surely (a.s.), 
in probability and in distribution, respectively, for random variables.

$a_m\sim b_m$, where $a_m$ and $b_m$ are real numbers, means
$a_m/b_m\to1$ as \mtoo. 
Furthermore, assuming $b_m>0$, 
$a_m=O(b_m)$ means $\sup_m |a_m|/b_m<\infty$
and $a_m=o(b_m)$ means $a_m/b_m\to0$ as \mtoo.
Moreover, if $X_m$ are random variables and $b_m>0$ are real numbers, then
$X_m=\Op(b_m)$ means that the sequence $|X_m|/b_m$ is stochastically bounded
(tight), i.e., $\sup_m\P(|X_m|/b_m>K)\to 0$ as $K\to\infty$,
and $X_m=\op(b_m)$ means $X_m/b_m\pto0$ as \mtoo.
We sometimes omit `as \mtoo' when clear from the context.

With high probability (\whp) means with probability $1-o(1)$ (as, \eg,
$n\to\infty$).

$C$ and $c$ are used for unspecified constants, which may vary from one
ocurrence to another. For constants that depend on some parameter (but not
on other variables such as $m$ or $n$), we similarly use \eg{} $C_K$ and
$c(\gd)$. 

If $x\in\bbR$, then 
$\floor{x}:=\max\set{n\in\bbZ: n\le x}$
and
$\ceil{x}:=\min\set{n\in\bbZ: n\ge x}$.

\section{General limit theorems}\label{Sgeneral}

We begin by formalising the hiring strategy discussed in \refS{S:intro}, 
at the same time introducing some further notation.
Recall that for a rank-based strategy, the result does not depend on the
(continuous) distribution of the values of the candidates.
We choose this distribution to be exponential.

Thus, let  $X_1,X_2,\dots$ be \iid{} random variables with
$X_i\xsim\Exp(1)$, representing the values of the candidates. 
We assume without further mention that these values are
distinct (which happens \as{}), so we ignore the possibility of ties below.
 For convenience we
identify a candidate and her value; we will thus say both 'candidate
$n$ is accepted' and 'value $X_n$ is accepted'.

Let $N_m$ be the index of the $m$-th accepted candidate,
and denote the $m$-th accepted  value by $\XX_m:=X_{N_m}$.
Conversely, let $M_n$ be the number of candidates accepted among $1,\dots,n$.
Thus, 
\begin{equation}
  \label{eq:NM}
%  M_n=m \iff N_m\le n<N_{m+1}.
M_n\ge m \iff N_m\le n.
\end{equation}

The hiring strategy is defined by a given function
$r:\bbZ_{\ge0}\to\bbZ_{>0}$ satisfying
%$r:\set{0,1,\dots}\to\set{1,2,\dots}$
\eqref{eq:rr}, and thus, in particular,
$1\le r(m)\le m+1$.

The basic rule of the strategy is that 
if $m\ge0$ values have been accepted so far, then the next value is
accepted if and only if it exceeds a threshold 
$Y_m$, which is the $r(m)$-th largest value among
the $m$ values already accepted, interpreted as $Y_m:=0$ when $r(m)=m+1$.
(In particular, $Y_0=0$.)

\begin{remark}
It is easy to see 
\cite{KriegerPS-rank2007}
that this threshold $Y_m$
is the same as the $r(m)$-th best value of all candidates seen so far, 
since all previous candidates with values at or above $Y_m$ were
accepted. If $r(m)\le m$, then the strategy is thus to accept a candidate if
her value is among the $r(m)$ best of all values seen so far (including her
own). It follows by symmetry that conditioned on $M_n=m$, and on everything
else that has happened earlier, the probability
that candidate $n+1$ is accepted equals $r(m)/(n+1)$, see
\cite{KriegerPS-rank2007}.
\end{remark}

The threshold $Y_m$ is thus updated when a new value is accepted. This is
described by the following lemma which is simple but basic for our analysis.
In particular, note that $Y_m$ never decreases.

\begin{lemma}
  \label{L1}
  \begin{romenumerate}
  \item 
If\/ $r(m+1)=r(m)+1$, then $Y_{m+1}=Y_m$.
\item     
If\/ $r(m+1)=r(m)$, then $Y_{m+1}$ is the smallest of the $r(m)$ values
that are larger than $Y_m$ among the selected values $\XX_1,\dots,\XX_{m+1}$.
  \end{romenumerate}
\end{lemma}
\begin{proof}
  Let $m$ be fixed and order the accepted values
 $\XX_1,\dots,\XX_m$
in decreasing order as $\XXo{1}>\XXo{2}>\dots>\XXo{m}$;  
define further $\XXo{m+1}:=0$. Then $Y_m=\XXo{r(m)}$.
By assumption, $\XX_{m+1}>Y_m$. 
Thus there are in the set $\set{\XX_1,\dots,\XX_{m+1}}$ 
exactly $r(m)$  values that are larger than $Y_m=\XXo{r(m)}$,
\viz{} $\set{\XXo{i}:1\le i< r(m)}\cup\set{\XX_{m+1}}$.
Hence, if $r(m+1)=r(m)$, then $Y_{m+1}$ is the smallest of these values,
while if $r(m+1)=r(m)$, then $Y_{m+1}$ is the next smaller accepted
value (or 0),
which is $\XXo{r(m)}=Y_m$.
\end{proof}

So far, the argument has been deterministic. We now use our assumption
that the values $X_i$ are \iid{} random variables as above; 
this is where the choice of exponential distribution 
is convenient and greatly simplifies the argument.

\begin{lemma}\label{L2}
  Assume as above that $X_1,X_2,\dots$ are \iid{} and $\Exp(1)$.
Then, the increments $Y_{m+1}-Y_m$, $m\ge0$, are independent random
variables with
\begin{equation}\label{l2}
  Y_{m+1}-Y_m\xsim
  \begin{cases}
    \Exp\bigpar{r(m)}, & r(m+1)=r(m), \\
0, & r(m+1)=r(m)+1.
  \end{cases}
\end{equation}
\end{lemma}

\begin{proof}
  Run the hiring process as above, but keep the values $X_n$ 
secret as long as possible, revealing only enough to determine whether to
accept $X_n$ or not, and to determine the next threshold $Y_m$.
To be precise, when a new candidate $n$ is
examined, reveal first only  whether her value $X_n$ 
is larger than the current threshold $Y_m$ or not. If not, we forget this
candidate and move on to the next.
Suppose instead that $X_n>Y_m$, so that we accept $n$.
Then we also have to update $Y_m$.
By \refL{L1}, this is trivial if $r(m+1)=r(m)+1$.
However, if $r(m+1)=r(m)$, then there are
$r(m)$ accepted candidates (including the latest recruit, $n$) that have
values $>Y_m$. We now reveal the minimum of these values, giving $Y_{m+1}$,
but we do not reveal the remaining $r(m)-1$ of them.
  
\begin{claim}
  Conditioned on $Y_m=y$ and on everything else
that has been revealed so far, the $r(m)-1$ 
(still hidden)
accepted values that are larger
than $Y_m$ 
have the distribution of $r(m)-1$ \iid{} random variables with
the distribution $\cL(X\mid X>y)$.
\end{claim}

To show the claim, we use induction on $m$.
We condition on $Y_m=y$ and everything else that has been
revealed so far, and 
note that when we accept the next $X_n$, we know just that
$X_n>y$, so $X_n$ too has the  distribution  $\cL(X\mid X>y)$.
Hence, by the induction hypothesis, the $r(m)$ accepted values that are
larger than $Y_m=y$ are $r(m)$ (conditionally)
independent random variables with this distribution.
By \refL{L1}, this completes the induction step when $r(m+1)=r(m)+1$;
otherwise we reveal the minimum $Y_{m+1}$ of them, and note that conditioned
on $Y_{m+1}=y'>y$, the remaining $r(m)-1=r(m+1)-1$ variables are \iid{} with
the distribution $\cL(X\mid X>y')$. 

This proves the claim.
Furthermore, 
since $X\xsim\Exp(1)$, this distribution $\cL(X\mid X>y)$
is the same as the distribution of $X+y$. 
(The standard lack-of-memory property of exponential distributions.)
Hence, if $r(m+1)=r(m)$, then
the claim and its proof yield that,
conditioned on $Y_m=y$ and everything else revealed so far,
\begin{equation}
Y_{m+1}= \min_{1\le j\le r(m)}  (E_j+y)
=Y_m + \min_{1\le j\le r(m)} E_j,
\end{equation}
where $E_1,E_2,\dots$ are \iid{} and $\Exp(1)$.
In particular, $Y_{m+1}-Y_m$ is independent of $Y_1,\dots,Y_m$.
Furthermore, \eqref{l2} holds, since
$
\min_{1\le j\le r(m)} E_j %\eqd r(m)\qw E_1
\xsim\Exp(r(m))$.
\end{proof}

Let
\begin{equation}
  \label{gd}
\gd_m:=\bigindic{r(m)=r(m-1)} 
%= 1-\bigpar{r(m)-r(m-1)}.
= 1+r(m-1)-r(m).
\end{equation}

\refL{L2} yields the following representation of $Y_m$.

\begin{lemma}\label{LY}
  There exists a sequence $E_1,E_2,\dots$ of \iid{} $\Exp(1)$ \rv{s} such
  that
  \begin{equation}
    \label{eq:ly}
    Y_m=\sum_{k=1}^m \frac{\gd_k}{r(k)} E_k,
\qquad m\ge0.
  \end{equation}
\end{lemma}

\begin{proof}
  An immediate consequence of \refL{L2}, with 
  \begin{equation}
E_k:=r(k)(Y_k-Y_{k-1})=r(k-1)(Y_k-Y_{k-1})\xsim\Exp(1)
  \end{equation}
  when $\gd_k=1$, and $E_k\xsim \Exp(1)$ arbitrary but independent of
  everything else when $\gd_k=0$.
\end{proof}

We are now prepared for a theorem giving an exact representation of the
sequence $N_m$.

\begin{theorem}\label{TN1}
The sequence $N_m$, $m\ge0$, is given by
  \begin{equation}\label{tn1}
    N_m=\sum_{k=0}^{m-1} V_k,
  \end{equation}
where, conditioned on the sequence $(Y_m)_1^\infty$ given by \refL{LY}, the
random variables
$V_k$ are independent with $V_k\xsim \Ge(e^{-Y_k})$.
\end{theorem}

\begin{proof}
Fix $m$ and condition on $Y_1,\dots,Y_m$ and $N_1,\dots,N_m$. 
Each new candidate after $N_m$ has probability 
$\P(X_n>Y_m\mid Y_m)=e^{-Y_m}$ of exceeding the
threshold $Y_m$, and these events are independent.
Hence, still conditioned on the past,
\begin{equation}\label{nge}
  N_{m+1}-N_m\xsim\Ge\bigpar{e^{-Y_m}}.
\end{equation}
Furthermore, still conditioned on the past, this waiting time $N_{m+1}-N_m$
is independent of the value of the next accepted candidate $\XX_{m+1}$.
Hence, the argument in the proof of \refL{L2} extends and shows that
conditioned on  $Y_1,\dots,Y_m$ and $N_1,\dots,N_m$,
the increments $Y_{m+1}-Y_m$ and $N_{m+1}-N_m$ are independent, with the
(conditional) distributions given by \eqref{l2} and \eqref{nge}.

This implies that conditioned on $(Y_m)\xoo$, the increments
$V_{k}:=N_{k+1}-N_{k}$ are independent, with (conditionally)
$V_k\xsim\Ge\bigpar{e^{-Y_{k}}}$.
\end{proof}

\begin{remark}
  As said above, our choice $X_n\xsim\Exp(1)$ simplifies the argument, but
  it is not really essential. An equivalent argument has been used by \eg{}
\cite{BroderEtAl}
with values $X_n\xsim U(0,1)$; then one considers the gap $1-Y_m$ and shows
that these gaps can be written as products of independent random variables. 
Taking (minus) the logarithm of the gap yields a sum of independent random
variables (which is more convenient than a product for limit theorems), and
that is equivalent to our version with exponentially distributed values $X_n$.
\end{remark}

So far we have given exact formulas, but now we start to approximate in
order to obtain simpler formulas.
First we approximate the geometric distributions in \eqref{tn1} by
exponential distributions.

\begin{theorem}
  \label{TN2}
As \mtoo, \as{}
\begin{equation}\label{tn2}
  N_m \sim \sum_{k=0}^{m-1} e^{Y_k} E'_k,
\end{equation}
where $Y_k$ are given by \eqref{eq:ly} and $E'_k\xsim\Exp(1)$ are independent
of each other and of $(Y_k)\xoo$.
\end{theorem}

\begin{proof}
  We use continuous time, and assume that candidate $n$ is examined at time
  $\tau_n$, where the waiting times $\tau_n-\tau_{n-1}$ (with $\tau_0=0$)
  are \iid{} $\Exp(1)$. In other words, the candidates arrive according to a
  Poisson process on $\ooo$ with intensity 1.
Note that by the law of large numbers, $\tau_n/n\asto1$ as \ntoo.
We assume that the times $(\tau_n)_n$ are independent of the values $(X_n)_n$.

Let $T_m:=\tau_{N_m}$ be the time the $m$-th candidate is accepted.
Then, as \mtoo{} and thus $N_m\to\infty$,
\begin{equation}
  \label{lil}
  \frac{T_m}{N_m}=\frac{\tau_{N_m}}{N_m}\asto1.
\end{equation}

We argue as in the proof of \refT{TN1}.
Condition on $Y_1,\dots,Y_m$ and $T_1,\dots,T_m$ for some $m$.
Then, after $T_m$, candidates arrive according to a Poisson process with
intensity 1, and thus candidates with a value $>Y_m$ arrive as a Poisson
process with intensity $e^{-Y_m}$. Consequently, conditioned on the past,
the waiting time $T_{m+1}-T_m$ has an exponential waiting time
\begin{equation}\label{tm2}
  T_{m+1}-T_m\xsim\Exp\bigpar{e^{-Y_m}}.
\end{equation}

By the same argument as in the proof of \refT{TN1}, this implies that conditioned on
$(Y_k)\xoo$, the  
increments $\hV_k:=T_{k+1}-T_k$ are independent with
$\hV_k\xsim\Exp\bigpar{e^{-Y_k}}$. Define $E'_k:=e^{-Y_k}\hV_k\xsim\Exp(1)$.
Then 
\begin{equation}\label{tm3}
  T_m=\sum_{k=0}^{m-1}\hV_k
=\sum_{k=0}^{m-1}e^{Y_k}E'_k.
\end{equation}
Furthermore, conditioned on $(Y_k)\xoo$, the variables $E'_k$ are \iid{}
$\Exp(1)$; hence, $(Y_k)\xoo$ and $(E'_k)\oxoo$ are independent.

Finally, the exact continuous-time representation \eqref{tm3} implies the
approximation \eqref{tn2} by \eqref{lil}.
\end{proof}

The results above are valid for any sequence $r(m)$ fulfilling the
conditions \eqref{eq:rr}.
For further approximations, we treat the cases of large and small $r(m)$
separately, in \refSs{Slarge}--\ref{Slarge++}
and \refS{Ssmall}, respectively.

We define, recalling \eqref{eq:ly}, 
\begin{equation}
  \label{eq:yk}
  y_m:=\E Y_m = \sum_{k=1}^m \frac{\gd_k}{r(k)}
= \sum_{\substack{1\le k\le m:\\ r(k)=r(k-1)}} \frac{\gd_k}{r(k)}
= \sum_{k=1}^m \frac{1}{r(k)} -
 \sum_{\ell=2}^{r(m)} \frac{1}{\ell},
\end{equation}
where the final equality follows because each $\ell\in\set{2,\dots,r(m)}$
equals $r(k)$ for exactly one $k\in\set{1,\dots,m}$ with $r(k)>r(k-1)$.

\section{Large $r(m)$}\label{Slarge}

In this section we assume that $r(m)$ is large, \ie,
\begin{equation}
  \label{eq:r2}
  \summ \frac{1}{r(m)^2}<\infty.
\end{equation}

\begin{lemma}
  \label{LZ}
Suppose that $\summ r(m)\qww<\infty$.
Then $Y_m-y_m\asto Z$ as \mtoo, where
\begin{equation}\label{Z}
Z:=\sum_{k=1}^\infty \frac{\gd_k}{r(k)} (E_k-1),
\end{equation}
with, as always, $(E_k)\xoo$ are \iid{} $\Exp(1)$.
\end{lemma}

\begin{proof}
  By \eqref{eq:ly} and \eqref{eq:yk},
  \begin{equation}\label{Yy}
Y_m-y_m=\sum_{k=1}^m \frac{\gd_k}{r(k)} (E_k-1),
  \end{equation}
so $Y_m-y_m$ converges to $Z$ given by \eqref{Z} whenever the latter sum
converges.
Furthermore, this occurs a.s., since the summands in \eqref{Z} are
independent random variables with mean 0 and sum of variances
\begin{equation}
  \sum_{k=1}^\infty\E\Bigpar{ \frac{\gd_k}{r(k)} (E_k-1)}^2
\le \sum_{k=1}^\infty \frac{1}{r(k)^2}<\infty.
\end{equation}
\end{proof}

We let $Z$ denote the sum in \eqref{Z} whenever \eqref{eq:r2} holds.

\begin{remark}
  Since the $\Exp(1)$ distribution is infinitely divisible with \Levy{}
  measure $x\qw e^{-x}\dd x$, it follows from \eqref{Z}
that $Z$ is infinitely divisible with \Levy{} measure,
arguing as in \eqref{eq:yk},
\begin{align}
  \sumk \gd_k  e^{-r(k) x} \frac{\dd x}x
=
\biggpar{ \sumk   e^{-r(k) x}-\frac{1}{e^{2x}-e^x}} \frac{\dd x}x,
\qquad x>0.
\end{align}
\end{remark}

\begin{theorem}
  \label{TN3}
Suppose that $\summ r(m)\qww<\infty$.
Then, 
as \mtoo, a.s.
\begin{equation}\label{tn3}
  N_m \sim e^Z\sum_{k=0}^{m-1} e^{y_k} E'_k,
\end{equation}
where $Z$ and $y_k$ are given by \eqref{Z} and \eqref{eq:yk}, 
and $E'_k\xsim\Exp(1)$ are independent of each other and of $Z$.
\end{theorem}

\begin{proof}
  Let $\eps_k:=Y_k-y_k-Z$; then a.s.\ $\eps_k\to0$ as \ktoo{} by \refL{LZ}.
Furthermore, $Y_k=y_k+Z+\eps_k$, and thus
\begin{equation}
  \sum_{k=0}^{m-1}e^{Y_k} E_k' =  e^Z \sum_{k=0}^{m-1}e^{\eps_k}e^{y_k} E_k'.
\end{equation}
Hence, the result \eqref{tn3} follows from \eqref{tn2} and the simple
deterministic Lemma \ref{LS} below, noting that $\sumko e^{y_k}E'_k\ge\sumko
E'_k=\infty$ a.s.
\end{proof}

\begin{lemma}
  \label{LS}
Suppose that $a_k>0$, $\sumko a_k=\infty$ and $\eps_k\to0$ as \ktoo.
Then, as \mtoo,
\begin{equation}\label{ls}
\frac{\sum_{k=0}^m (1+\eps_k)a_k}{ \sum_{k=0}^m a_k} \to1.
%\sum_{k=0}^m (1+\eps_k)a_k \bigm/ \sum_{k=0}^m a_k \to1.
\end{equation}
\end{lemma}
\begin{proof}
  Let $\eta>0$. Then there exists $K$ such that if $k>K$, then
  $|\eps_k|<\eta$.
Consequently, for $m>K$,
\begin{equation}%\label{ls}
  \lrabs{\sum_{k=0}^m \eps_ka_k}
\le
\sum_{k=0}^K|\eps_k|a_k
+
\eta \sum_{k=K+1}^m a_k,
\end{equation}
which is less than $2\eta \sum_{k=0}^m a_k$ if $m$ is large enough.
This implies \eqref{ls}.
\end{proof}

\begin{lemma}\label{L4}
Suppose that
\begin{equation}
  \label{eq:r2b}
  \summ \frac{1}{r(m)^2}<\infty.
\end{equation}
Then 
\begin{equation}
\summo  \frac{e^{2y_m}}{\Bigpar{\sum_{k=0}^m e^{y_k}}^2}
<\infty.
\end{equation}
\end{lemma}

\begin{proof}
Let $m\ge0$ and let $a(m):=r(\floor{m/2})\le \floor{m/2}+1$.
If $0\le i < a(m)$, then $i\le a(m)-1\le m/2$, and thus $m-i\ge m/2$.
Hence, 
\begin{equation}
  y_{m}-y_{m-i}=\sum_{j=m-i+1}^m\frac{\gd_j}{r(j)} 
\le  \frac{i}{r(\ceil{m/2})}
\le \frac{i}{a(m)}
\le 1.
\end{equation}
Consequently, 
\begin{equation}
  \label{eq:b}
  \sum_{k=0}^m e^{y_k}
\ge
  \sum_{i=0}^{a(m)-1} e^{y_{m-i}}
\ge
  \sum_{i=0}^{a(m)-1} e^{y_{m}-1}
=a(m)e^{y_m-1}
\end{equation}
and thus
\begin{equation}
\summo  \frac{e^{2y_m}}{\Bigpar{\sum_{k=0}^m e^{y_k}}^2}
\le
%\summo  \frac{e^{2y_m}}{a(m)^2e^{2(y_m-1)}}= 
\summo  \frac{e^{2}}{a(m)^2}
=2\sum_{\ell=0}^\infty\frac{e^2}{r(\ell)^2}
<\infty.
\end{equation}
\end{proof}

Define, with $y_k$  given by \eqref{eq:yk}, 
\begin{equation}\label{gl}
  \gl_m:=\sumkomi e^{y_k}, \qquad m\ge0.
\end{equation}

\begin{theorem}
  \label{TN4}
Suppose that $\summ r(m)\qww<\infty$.
Then, 
as \mtoo, 
\begin{equation}\label{tn4}
%  N_m  \sim \gl_m e^Z,
  N_m / \gl_m \asto e^Z,
\end{equation}
where $Z$ and $\gl_m$ are given by \eqref{Z} and \eqref{gl}.
\end{theorem}

\begin{proof}
  Let $W_m:=e^{y_m}(E'_m-1)$ and $b_m:=\gl_{m+1}$.
Then $\E W_m=0$ and \refL{L4} shows that $\summo \Var(W_m)/b_m^2<\infty$.
Consequently,
see \cite[Theorem VII.8.3]{FellerII},
$b_m\qw\sum_{k=0}^m W_k \asto0$.
Hence,
\begin{equation}
\frac{\sum_{k=0}^{m-1} e^{y_k} E'_k} {\gl_m}
=1+ \frac{\sum_{k=0}^{m-1} e^{y_k} (E'_k-1)} {\gl_m}
=1+ \frac{\sum_{k=0}^{m-1} W_k} {b_{m-1}}\asto1,
\end{equation}
and the result follows from \eqref{tn3}.
\end{proof}

Equivalently, \as{}
$N_m  \sim \gl_m e^Z$  as \mtoo. 
Hence, $N_m$ grows as the deterministic sequence $\gl_m$, with
a random factor (asymptotically independent of $m$) given by $e^Z$.

\refT{TN4} gives the asymptotics (and in particular the asymptotic
distribution) of $N_m$, the number of candidates examined until $m$ have
been accepted. By inversion, we obtain corresponding asymptotic results for
$M_n$, the number of accepted candidates when $n$ have been examined.
We state one general result as the next theorem. 
More explicit results require inversion of
the function $m\mapsto\gl_m$, which easily is done under further
assumptions; we study an important case in \refS{Slinear} below.

Note first
that $M_n\asto\infty$ as \ntoo, since for every $m$, \as{} some future
candidate $n>N_m$ will satisfy $X_n>Y_m$ and thus be accepted.

\begin{theorem}
  \label{TM1}
Suppose that $\summ r(m)\qww<\infty$.
Then, 
as \ntoo,
\begin{equation}
  \gl_{M_n}/n \asto e^{-Z},
\end{equation}
where $Z$ and $\gl_m$ are given by \eqref{Z} and \eqref{gl}.
\end{theorem}

\begin{proof}
Since $M_n\to\infty$ \as{} as \ntoo, 
\eqref{tn4} and \refL{Lgl} below imply that
\as{} $N_{M_n}/\gl_{M_n}\asto e^Z$ and $N_{M_n+1}/\gl_{M_n}\asto e^Z$.
Furthermore, the definitions imply $N_{M_n}\le n< N_{M_n+1}$.
Hence, $n/\gl_{M_n}\asto e^Z$.
\end{proof}

\begin{lemma}\label{Lgl}
If $r(m)\to\infty$ 
  as \mtoo, then $\gl_{m+1}/\gl_m\to1$.
\end{lemma}
\begin{proof}
  By \eqref{gl} and \eqref{eq:b},
  \begin{equation}
    \gl_{m+1}-\gl_m = e^{y_m} 
\le \frac{e}{a(m)}\sum_{k=0}^me^{y_k}
=\frac{e}{a(m)}\gl_{m+1}.
  \end{equation}
Since $a(m):=r(\floor{m/2})\to\infty$ as \mtoo{} by assumption, it follows that
$(\gl_{m+1}-\gl_m)/\gl_{m+1}\to0$, and thus $\gl_m/\gl_{m+1}\to1$.
\end{proof}

\section{Roughly linear rank thresholds}\label{Slinear}

As said in the introduction, the strategies  
\hmedian{} and \hpercentile{\ga}
satisfy $r(m)=\ga m+O(1)$ for some constant $\ga>0$, and we stated
\refT{TI1} for this case.
Note that $r(m)=\ga m+O(1)$ 
implies $r(m)\qw = (\ga m)\qw + O(1/m^2)$, and thus
\begin{equation}\label{sofie}
  \summ \bigabs{r(m)\qw - (\ga m)\qw}<\infty.
\end{equation}
In fact, by the proof below, \refT{TI1} holds under the weaker
assumption \eqref{sofie}.
(For example, if $r(m)=\ga m + O(m^{1-\eta})$
for some $\eta>0$.)

 We assume throughout this section that \eqref{sofie} holds, with
$0<\ga\le1$.
We then define
\begin{equation}\label{rho}
\rho:=  \summ \Bigpar{\frac{1}{r(m)} - \frac{1}{\ga m}}\in\bbR.
\end{equation}
 Note that \eqref{sofie} implies
\begin{equation}\label{sofie2}
  r(m)\sim \ga m
\qquad \text{as \mtoo}.
\end{equation}
This is an easy exercise, but for the reader's convenience we give a proof in
 \refApp{A1}.
Note also that \eqref{sofie2} implies $\summ r(m)\qww<\infty$, so $r(m)$ is
large and we can use the results of \refS{Slarge}; our goal in this section
is to use the assumption \eqref{sofie} to make the results more explicit.

\begin{lemma}
  \label{LA}
Suppose that \eqref{sofie} holds for some $\ga\in(0,1]$. Then, 
as \mtoo,
\begin{equation}\label{la}
  y_m = \bigpar{\ga\qw-1}\log m + \bigpar{\ga\qw-1}\gam-\log\ga+\rho+1+o(1)
.
\end{equation}
\end{lemma}

\begin{proof}
  Let $H_m:=\sum_1^m 1/k$, the $m$-th harmonic number, and recall
that $H_m=\log m + \gam + o(1)$.
Hence, by \eqref{eq:yk} and \eqref{rho},
\begin{align}
y_m
&= 
\sum_{k=1}^m \Bigpar{\frac{1}{r(k)}-\frac{1}{\ga k}} + \frac{1}{\ga}H_m 
 -(H_{r(m)} -1)
\nonumber\\
&=\rho + \ga\qw\bigpar{\log m + \gam} -\log r(m)-\gam +1 + o(1).
\end{align}
The result \eqref{la} follows since $\log(r(m))=\log(\ga m)+o(1)
=\log m + \log\ga+o(1)$ by \eqref{sofie2}.
\end{proof}

\begin{lemma}
  \label{LA2}
Suppose that \eqref{sofie} holds for some $\ga\in(0,1]$. Then, as \mtoo,
\begin{equation}\label{la2}
  \gl_m \sim e^{\xpar{\ga\qw-1}\gam+\rho+1} m^{1/\ga}.
\end{equation}
\end{lemma}
\begin{proof}
\refL{LA} implies, where $o(1)\to0$ as \ktoo,
  \begin{equation}\label{la2a}
\gl_m:=
\sum_{k=0}^{m-1}e^{y_k}
= e^{\xpar{\ga\qw-1}\gam+\rho+1}\ga\qw\sum_{k=0}^{m-1} k^{\ga\qw-1}\etto,
\end{equation}
and the result \eqref{la2} follows.
\end{proof}

We are prepared to prove the convergence \eqref{limW} in \refT{TI1}.

\begin{lemma}
  \label{LB}
Suppose that \eqref{sofie} holds for some $\ga\in(0,1]$. Then, as \ntoo,
\begin{equation}\label{W}
    M_n/n^\ga \asto W:=e^{(\ga-1)\gam- \ga\rho-\ga} e^{-\ga Z},
\end{equation}
with $\rho$ and $Z$ given by \eqref{rho} and \eqref{Z}.
\end{lemma}
\begin{proof}
  \refT{TM1} and \refL{LA2} imply that a.s.
  \begin{equation}
    M_n^{1/\ga}/n \to e^{-(\ga\qw-1)\gam-\rho-1} e^{-Z},
  \end{equation}
and thus \eqref{W} follows.
\end{proof}

We proceed to compute moments of $W$.

\begin{lemma}
  \label{LZ1}
Suppose that \eqref{sofie} holds for some $\ga\in(0,1]$.
If\/ $-\infty < u <\ga$, then
\begin{equation}
  \label{lz1}
\E e^{uZ} = e^{-u\rho-u+u(1-\ga\qw)\gam}
\frac{\gG(1-u/\ga)}{\gG(2-u)}\prodk\frac{1-u/\ga k}{1-u/r(k)}
<\infty.
\end{equation}
Hence, at least for $-1 < s <\infty$,
\begin{align}\label{ti1=}
\E W^s 
&= \frac{\gG(s+1)}{\gG(s\ga+2)} \prodk \frac{1+s/k}{1+s\ga /r(k)}
%\\
%&= \frac{\gG(s+1)}{\gG(s\ga+2)} \prodk \frac{r(k)+sr(k)/k}{r(k)+s\ga}
.
\end{align}
\end{lemma}

\begin{proof}%[Proof of \refL{LZ1}]
It is elementary that $\E e^{uE_k} =1/(1-u) $ for $u<1$.
Hence, \eqref{Z} implies
\begin{equation}\label{qn}
  \E e^{uZ} = \prod_{k\ge1:\,\gd_k=1} \frac{e^{-u/r(k)}}{1-u/r(k)}.
\end{equation}
Arguing as in \eqref{eq:yk}, this can be written
\begin{align}\label{qm}
  \E e^{uZ}
&=\prodk \frac{e^{-u/r(k)}}{1-u/r(k)} 
  \prod_{\ell=2}^\infty \frac{1-u/\ell} {e^{-u/\ell}}
  \end{align}
where the products are absolutely convergent as a consequence of
\eqref{sofie2} and simple estimates. 
Using \eqref{rho}, we rewrite this as
  \begin{align}\label{ql}
    \E e^{uZ}
=
e^{-u\rho}
\prodk \frac{e^{-u/\ga k}}{1-u/\ga k} 
\prodk \frac{1-u/\ga k}{1-u/r(k)} 
  \prod_{\ell=2}^\infty \frac{1-u/\ell} {e^{-u/\ell}}.
\end{align}
The standard product formula for the Gamma function
\cite[(5.8.2)]{NIST} can be written
$\prodk e^{z/k}/(1+z/k) =z e^{\gam z}\gG(z)=e^{\gam z}\gG(z+1)$,
and thus \eqref{ql} yields
\begin{align*}
  \E e^{uZ}
&=
e^{-u\rho}
e^{-u\ga\qw \gam} \gG(1-u/\ga) \frac{e^{-u}}{1-u} 
e^{u \gam} \gG(1-u)\qw
\prodk \frac{1-u/\ga k}{1-u/r(k)} 
, 
\end{align*}
which simplifies to \eqref{lz1}.

Finally, \eqref{ti1=} follows from \eqref{W} and \eqref{lz1}, with $u=-\ga s$.
\end{proof}
\begin{remark}\label{RZ1} 
Let $\rx:=\inf\set{r(k):\gd_k=1}\ge1$ be the first repeated value in the
sequence $r(0),r(1),\dots$.
Then the proof shows that \eqref{lz1}
 holds for $\ga\le u<\rx$ too, with the right hand side interpreted
by continuity when necessary, and that $\E e^{uZ}=\infty$  for $u\ge \rx$.
Consequently, \eqref{ti1=} holds for $s\ge-\rx/\ga$.
Furthermore, by analytic continuation, or by the proof above, \eqref{ti1=}
extends to complex $s$ with $\Re s>-\rx/\ga$.
However, $\E W^{-\rx/\ga}=\infty$.
\end{remark}

Next we bound moments of $M_n$, using a series of lemmas.
Recall $T_m:=\tau_{N_m}$ from the proof of \refT{TN2}.
We tacitly continue to assume \eqref{sofie}.

\begin{lemma}
  \label{LtT}
For every integer
$K\ge1$ there exists a constant $C_K$ such that if $m\ge 2K+2$,
then
\begin{equation}\label{ltt}
\E\bigpar{T_m/m^{1/\ga}}^{-K} \le C_K.
\end{equation}
\end{lemma}

\begin{proof}
For convenience, we first consider $T_{2m}$. By \eqref{tm3},
\begin{equation}
  T_{2m}\ge \sum_{k=m}^{2m-1} e^{Y_k}E_k' 
\ge e^{Y_m}\sum_{k=m}^{2m-1} E_k'=: e^{Y_m}U,
\end{equation}
where $U:=\sum_{k=m}^{2m-1} E_k'$ has a Gamma distribution $\Gamma(m)$ 
and is independent of $Y_m$.
Consequently,
\begin{equation}\label{tre}
  \E T_{2m}^{-K} \le \E e^{-K (Y_m-y_m)} e^{-K y_m} \E U^{-K}.
\end{equation}
We estimate the three factors on the \rhs{} separately.

For the first factor, by \eqref{Yy} and \eqref{qn}, for any real $u$,
\begin{equation}\label{Q0}
    \E e^{u(Y_m-y_m)} = \prod_{k\le m:\,\gd_k=1} \frac{e^{-u/r(k)}}{1-u/r(k)}
\le 
  \E e^{uZ},
\end{equation}
since each factor in \eqref{qn} is $\ge1$ (by the explicit form or by
Jensen's inequality). In particular, 
\begin{equation}\label{Q1}
\E e^{-K(Y_m-y_m)}\le \E e^{-KZ}=C_K,
\end{equation}
using \eqref{lz1} to see that  $\E e^{-KZ}<\infty$.

For the second factor in \eqref{tre}, \refL{LA} yields
\begin{equation}\label{Q2}
  e^{-K y_m} = m^{-K(\ga\qw-1)} C_K e^{o(1)}
\le m^{-K(\ga\qw-1)} C_K.
\end{equation}

Finally, for the third factor, we use the fact that $U\xsim\gG(m)$ and compute
\begin{equation}\label{Q3}
  \E U^{-K} =\frac{1}{\gG(m)}  \intoo x^{-K} x^{m-1} e^{-x} \dd x
=\frac{\gG(m-K)}{\gG(m)}
\le C_K m^{-K},
\end{equation}
for $m\ge K+1$.

Combining \eqref{Q1}, \eqref{Q2} and \eqref{Q3} with \eqref{tre}, we find
that if $m\ge K+1$, then
\begin{equation}
  \E T_{2m}^{-K} \le C_K m^{-K/\ga}.
\end{equation}
This proves \eqref{ltt} for even $m\ge 2K+2$, and the case of odd $m$
follows because $T_{2m+1}\ge T_{2m}$.
\end{proof}

\begin{lemma}
  \label{LtN}
For every
$K>0$ there exists a constant $C_K$ such that 
%if $m\ge 1$, then
\begin{equation}\label{ltn}
\E\bigpar{N_m/m^{1/\ga}}^{-K} \le C_K,
\qquad m\ge1.
\end{equation}
\end{lemma}

\begin{proof}
We may assume that $K$ is an integer (by Lyapunov's inequality).
Furthermore,  $N_m\ge m\ge1$, and thus $\E N_m^{-K}\le1$.
It follows that \eqref{ltn} holds trivially for $m< 4K+2$, so we may assume
$m\ge 4K+2$.
In this case, we use the \CSineq{} and obtain by \refL{LtT} 
\begin{equation}\label{qq}
\lrpar{  \E \Bigparfrac{N_m}{m^{1/\ga}}^{-K}}^2
\le
  \E \Bigparfrac{T_m}{m^{1/\ga}}^{-2K}
  \E \Bigparfrac{N_m}{T_m}^{-2K}
\le C_K   \E \Bigparfrac{T_m}{N_m}^{2K}.
\end{equation}
Furthermore, conditioned on $N_m=n$, $T_m=\tau_n$ is the sum of $n$ \iid{}
waiting times $E''_i\xsim\Exp(1)$. Since $E''_i$ have moments of all orders, 
the law of large numbers holds with moment convergence 
\cite[Theorem 6.10.2]{Gut}, and thus $\E \tau_n^{2K}/n^{2K}\to1$ as $\ntoo$,
and therefore $\E \tau_n^{2K}/n^{2K}\le C_K$ for all $n\ge1$.
Consequently,
\begin{equation}
   \E \Bigpar{\Bigparfrac{T_m}{N_m}^{2K}\Bigm| N_m=n}
=\frac{\E \tau_n^{2K}}{n^{2K}}\le C_K
\end{equation}
and thus
\begin{equation}\label{qa}
   \E \Bigparfrac{T_m}{N_m}^{2K}
=    \E \Bigpar{\Bigparfrac{T_m}{N_m}^{2K}\Bigm| N_m}
\le C_K.
\end{equation}
The result follows by \eqref{qq} and \eqref{qa}.
\end{proof}

\begin{lemma}
  \label{LtM}
For every 
$K>0$ there exists a constant $C_K$ such that 
%if $n\ge 1$, then
\begin{equation}\label{ltm}
\E\bigpar{M_n/n^{\ga}}^{K} \le C_K,
\qquad n\ge1.
\end{equation}
\end{lemma}

\begin{proof}
  Let $x>0$ and $n\ge1$, and let $m:=\ceil{x n^{\ga}}$.
If $M_n\ge x n^{\ga}$, then $M_n\ge m$, and thus, by \eqref{eq:NM},
 $N_m\le n \le x^{-1/\ga}m^{1/\ga}$.
Consequently, for any $\KK>0$, by Markov's inequality and \refL{LtN},
\begin{equation}
  \P \bigpar{M_n\ge x n^{\ga}}
\le 
  \P \bigpar{N_m\le x^{-1/\ga} m^{1/\ga}}
\le 
x^{-\KK/\ga}
\E\bigpar{N_m/m^{1/\ga}}^{-\KK} 
\le C_{\KK} x^{-\KK/\ga}.
\end{equation}
Choosing $\KK:=(K+1)\ga$, we obtain
$  \P \bigpar{M_n\ge x n^{\ga}}\le C_K x^{-(K+1)}$, which implies \eqref{ltm}.
\end{proof}

\begin{proof}[Proof of \refT{TI1}]
We have shown the \as{} convergence \eqref{limW} in Lemma \ref{LB}, and moment
convergence follows from this and the uniform estimate in \refL{LtM}.
Finally, the moments of $W$ are computed in \refL{LZ1}.
\end{proof}

Although perhaps of less interest, we show further
that the moment convergence in
\refT{TI1} holds also for some, but not all, $s<0$.
Let $\rx\ge1$ be as in \refR{RZ1}.

\begin{lemma}
  \label{LtN+}
For every real $u<\rx$, there exists a constant $C_u$ such that 
%if $m\ge 1$, then
\begin{equation}
%  \E\bigpar{T_m/m^{1/\ga}}^{u} &\le C_u,\label{ltt+}\\
\E\bigpar{N_m/m^{1/\ga}}^{u} \le C_u,
\qquad m\ge1.\label{ltn+}
\end{equation}
\end{lemma}

\begin{proof}
The case $u<0$ is \refL{LtN}, and $u=0$ is trivial, so we may assume $u>0$.
We consider again first $T_m=\tau_{N_m}$.

Assume first that $r(1)=2$, so that $\rx\ge2$. It then suffices to prove
\eqref{ltn+} for $1\le u<\rx$, so we assume this.
Recall from \refR{RZ1} that then $\E e^{uZ}<\infty$.
Hence, \eqref{tm3}, Minkowski's inequality, \eqref{Q0}, \eqref{gl} and
\eqref{la2} yield, 
with $\norm{X}_u:=\xpar{\E |X|^u}^{1/u}$,
\begin{align}\label{kf}
 \norm{T_m}_u 
&\le
\sum_{k=0}^{m-1}\bignorm{e^{Y_k}E'_k}_u
=\sum_{k=0}^{m-1} e^{y_k}\bigpar{\E e^{u(Y_k-y_k)}}^{1/u} \norm{E'_k}_u
\\&
\le
C_u\bigpar{\E e^{u Z}}^{1/u}\sum_{k=0}^{m-1}  e^{y_k}
= C_u \gl_m \le C_u m^{1/\ga}.
\end{align}
In other words, 
\begin{align}\label{kf2}
  \E T_m^u \le
C_u  m^{u/\ga}, \qquad m\ge1.
\end{align}

The argument above fails if $r(1)=1$, since then $\rx=1$ and
$\E e^{uZ}=\infty$ for every $u\ge1$, see again \refR{RZ1}. 
(And we need $u\ge1$ in
order to use Minkowski's inequality.)
In this case, 
let $k_0:=\min\set{k:r(k)=2}$ and consider
\begin{equation}
  \tT_m:=\sum_{k=k_0}^{m-1} e^{Y_k-Y_{k_0}} E'_k, \qquad m\ge1.
\end{equation}
We have, \cf{} \eqref{Yy} and \eqref{qn}, 
\begin{equation}
  \E e^{Y_k-Y_{k_0}-(y_k-y_{k_0})} 
= 
\prod_{k_0< j \le m:\,\gd_j=1} \frac{e^{-1/r(j)}}{1-1/r(j)}
\le 
\prod_{j=k_0}^\infty\frac{e^{-1/r(j)}}{1-1/r(j)}<\infty,
\end{equation}
and thus, using \eqref{gl} and \eqref{la2} as above,
\begin{equation}
 \E \tT_m:=\sum_{k=k_0}^{m-1} \E e^{Y_k-Y_{k_0}} 
\le C \sum_{k=k_0}^{m-1} e^{y_k-y_{k_0}} 
\le C\gl_m
\le C m^{1/\ga}.
\end{equation}
Consequently, for $0<u<1=\rx$, using $T_m=T_{k_0}+e^{Y_{k_0}}\tT_m$
and the subadditivity of $x\mapsto x^u$,
\begin{align}
    \E T_m^u 
&\le \E T_{k_0}^u + \E e^{uY_{k_0}} \E \tT_m^u
\le \sum_{k=0}^{k_0-1}\E e^{uY_k} \E (E'_k)^u 
  + \E e^{uY_{k_0}} \bigpar{\E \tT_m}^u
\notag\\&
\le C_u+C_u m^u.
\end{align}
Hence, \eqref{kf2} holds in this case too.

Finally, by the law of large numbers, 
$\P(\tau_n>n/2)\to1$ as \ntoo, and thus $\P(\tau_n>n/2)\ge c$ for every $n\ge1$.
Hence,
\begin{equation}
  \E\bigpar{T_m^u\mid N_m} \ge c \xpar{N_m/2}^u = c_u N_m^u
\end{equation}
and thus
\begin{equation}
\E T_m^u \ge c_u \E N_m^u.
\end{equation}
Consequently, \eqref{ltn+} follows from \eqref{kf2}.
\end{proof}

\begin{lemma}
  \label{LtM-}
For every $s>-\rx/\ga$, 
\begin{equation}\label{ltm-}
\E\bigpar{M_n/n^{\ga}}^{s} \le C_s.
\end{equation}
\end{lemma}

\begin{proof}
By \refL{LtM}, it suffices to consider $s<0$. Then, let $-\rx/\ga<t<s$ and
$u:=-\ga t\in (0,\rx)$.
We argue as in the proof of \refL{LtM} with minor modifications. Let $x>0$
and let $m:=\floor{x n^\ga}$.
Then, by \refL{LtN+},
\begin{align}
  \P\bigpar{M_n+1 < x n^{\ga}}
&
\le   \P\bigpar{M_n < m} = \P(N_m>n)
\le n^{-u} \E N_m^{u} 
\nonumber\\&
\le C_u m^{u/\ga} n^{-u} \le C_u x^{u/\ga} = C_u x^{-t}.
\end{align}
It follows that, since trivially $1+M_n\le 2M_n$,   
\begin{align}
  \E M_n^s \le 2^{-s}\E (M_n+1)^s \le C_s n^{s\ga}.
\end{align}
\end{proof}

\begin{theorem}\label{TM-}
  The moment convergence \eqref{ti1} holds for every real $s>-\rx/\ga$, and
  more generally for every complex $s$ with $\Re s >-\rx/\ga$, with a finite
  limit. 
On the other hand, if $s\le -\rx/\ga$, then $\E M_n^s\to\infty$.
\end{theorem}

\begin{proof}
Note that the \rhs{} of \eqref{ti1} is finite for $\Re s>-\rx/\ga$
and an analytic function of $s$ in that half-plane, see \refR{RZ1}.

\refL{LtM-} implies uniform integrability of $(M_n/n^{\ga})^s$ for every
$s\in(-\rx/\ga,0)$, and thus by \refT{TI1} for every  real $s>-\rx/\ga$,
which implies \eqref{ti1} for $\Re s>-\rx/\ga$.
On the other hand, if $s\le -\rx/\ga$, then $\E W^s=\infty $, see \refR{RZ1},
and thus $\E M^s/n^{\ga s}\to\infty$ by \eqref{limW} and Fatou's lemma.
\end{proof}

\section{Linear-periodic rank thresholds}\label{Slarge++}
In this section we restrict $r(m)$ further to an important case
where we can evaluate the moments in \eqref{ti1=} explicitly.
We assume that 
\begin{equation}\label{per}
  r(m+q) = r(m)+\nu, \qquad m\ge1,
\end{equation}
for some integers $\nu$ and $q$
with $1\le\nu\le q$.
Hence, with $\ga:=\nu/q\in\ooi$, $r(m)-\ga m$ is periodic. 
In particular, $r(m)=\ga m+O(1)$, 
and thus the results of \refSs{Slarge}--\ref{Slinear} apply.
Moreover, it is obivious from \eqref{hmedian}--\eqref{hpercentile} that
\hmedian{} satisfies \eqref{per} with $\nu=1$, $q=2$, and that 
\hpercentile{\ga} satisfies \eqref{per} whenever $\ga=\nu/q$ is rational.

When \eqref{per} holds, the asymptotic moments can be evaluated explicitly;
for convenience in applications, we give several equivalent formulas. 
\begin{theorem}\label{Tper}
  Suppose that \eqref{per} holds, with $1\le\nu\le q$, and let $\ga:=\nu/q$.
Then,  for $0\le s<\infty$, as \ntoo,
\begin{align}\label{tomas}
  \E M_n^s/n^{\ga s} \to 
\E W^s
&=
%&\frac{\gG(s+1)}{\gG(s\ga+2)} \prodk \frac{1+s/k}{1+s\ga /r(k)}
%\\&= 
\frac{\gG(s+1)}{\gG(s\ga+2)} 
\prodiq \frac{\gG\bigpar{\frac{i}q}\gG\bigpar{\frac{s}{q}+\frac{r(i)}\nu}}
 {\gG\bigpar{\frac{r(i)}\nu}\gG\bigpar{\frac{s}{q}+\frac{i}q}}
\\&\label{tomas2}
=
\frac{ q^{s}}{\gG(s\ga+2)} 
\prodiq \frac{\gG\bigpar{\frac{s}{q}+\frac{r(i)}\nu}}
 {\gG\bigpar{\frac{r(i)}\nu}}
\\&\label{tomas3}
=
\frac{ q^{s}}{\nu^{\ga s}}
\prodiq \frac{\gG\bigpar{\frac{s}{q}+\frac{r(i)}\nu}} 
  {\gG\bigpar{\frac{r(i)}\nu}}
\prod_{j=2}^{\nu+1} \frac
{\gG\bigpar{\frac{j}{\nu}}}
{\gG\bigpar{\frac{s}{q}+\frac{j}{\nu}}}
\\&\label{tomas4}
=
\frac{ q^{s+1}}{\nu^{\ga s}(\nu s+q)}
\prodiq \frac{\gG\bigpar{\frac{s}{q}+\frac{r(i)}\nu}} 
  {\gG\bigpar{\frac{r(i)}\nu}}
\prod_{j=1}^{\nu} \frac
{\gG\bigpar{\frac{j}{\nu}}}
{\gG\bigpar{\frac{s}{q}+\frac{j}{\nu}}}
.
%<\infty.
% \\
%   \gG(nz) = (2\pi)^{(1-n)/2} n^{nz-1/2} \prod_{k=0}^{n-1}
%   \gG\Bigpar{z+\frac{k}{n}}.
% \\
%   \gG(s+1) = (2\pi)^{(1-q)/2} q^{s+1/2} \prod_{k=1}^{q}
%   \gG\Bigpar{\frac{s}{q}+\frac{k}{q}}.
% \\
%   \gG(s\ga+2) = (2\pi)^{(1-\nu)/2} \nu^{s\ga+3/2} \prod_{k=0}^{\nu-1}
%   \gG\Bigpar{\frac{s}{q}+\frac{2}{\nu}+\frac{k}{\nu}}.
\end{align}
\end{theorem}

\begin{proof}
\refT{TI1} applies and shows \eqref{ti1}, and it remains only to compute the
infinite product there.
Standard manipulations yield, for every $s\ge0$,
writing $k=jq+i$ and noting that $r(jq+i)=j\nu+r(i)$,
\begin{align*}
& \prodk \frac{1+s/k}{1+s\ga /r(k)}
=
 \prodk \frac{(k+s)r(k)}{k(r(k)+s\ga)}
=\prodiq\prodoj  \frac{(jq+i+s)r(jq+i)}{(jq+i)\bigpar{r(jq+i)+s\ga}}
\\&
\hskip3em
\begin{aligned}
&
=\prodiq\prodoj  \frac{(jq+i+s)(j\nu+r(i))}{(jq+i)\bigpar{j\nu+r(i)+s\ga}}
=\prodiq\prodoj  \frac{\bigpar{j+\frac{i+s}q}\bigpar{j+\frac{r(i)}\nu}}
   {\bigpar{j+\frac{i}q}\bigpar{j+\frac{r(i)}\nu+\frac{s}q}}
\\&
=\prodiq\lim_{n\to\infty}%\prod_{j=0}^{n-1}  
\frac{\gG\bigpar{n+\frac{i+s}q}\gG\bigpar{n+\frac{r(i)}\nu}}
 {\gG\bigpar{\frac{i+s}q}\gG\bigpar{\frac{r(i)}\nu}}
\frac{\gG\bigpar{\frac{i}q}\gG\bigpar{\frac{r(i)}\nu+\frac{s}q}}
    {\gG\bigpar{n+\frac{i}q}\gG\bigpar{n+\frac{r(i)}\nu+\frac{s}q}}
\\&
=\prodiq \frac{\gG\bigpar{\frac{i}q}\gG\bigpar{\frac{r(i)}\nu+\frac{s}{q}}}
 {\gG\bigpar{\frac{i+s}q}\gG\bigpar{\frac{r(i)}\nu}}
.
\end{aligned}
\end{align*}
The result \eqref{tomas} follows.
Furthermore, by the Gauss multiplication formula for the Gamma function
\cite[(5.5.6)]{NIST},
\begin{align}\label{tho1}
\gG(s+1) &= (2\pi)^{(1-q)/2} q^{s+1/2} 
  \prod_{k=0}^{q-1}  \gG\bigpar{\frac{s}{q}+\frac{1}q+\frac{k}{q}},
  \\\label{tho2}
   \gG(s\ga+2)& = (2\pi)^{(1-\nu)/2} \nu^{s\ga+3/2} \prod_{k=0}^{\nu-1}
   \gG\Bigpar{\frac{s}{q}+\frac{2}{\nu}+\frac{k}{\nu}},
\end{align}
which together with the special case $s=0$ in \eqref{tho1} and \eqref{tho2}
easily yield \eqref{tomas2}--\eqref{tomas4}.
\end{proof}

See \refEs{Emedian}, \ref{Epercentile} and \ref{Ekapad}.

\begin{example}
  Taking $s=q$ in \eqref{tomas2}, we see that the $q$-th moment has the
  rational value 
  \begin{align}
\E W^q=
\frac{q^q}{\gG(\nu+2)} 
\prodiq \frac{r(i)}\nu
    =
\frac{q^q} {\nu^q(\nu+1)!}\prodiq r(i).
  \end{align}
\end{example}

\begin{remark}
  The result extends
to all complex $s$ such
  that $\Re s>-q/\nu$ (and more generally to $\Re s>-\rx/\ga$); 
in particular to imaginary $s$, which gives the
  \chf{} of $\log W$, see \refR{RZ1} and \refT{TM-}.
\end{remark}

\begin{remark}
  \label{R244}
Positive random variables with moments that can be expressed as a fraction
of products of Gamma functions as in \eqref{tomas} are studied in \eg{}
\cite{SJ244}. In particular,
\cite[Theorems 5.4 and  6.1]{SJ244} imply
that $W$ has an infinitely differentiable density dunction $f_W(x)$ on $\ooo$,
with the asymptotic 
\begin{equation}
  \label{eq:fW}
  f_W(x)\sim C_2  x^{c_1-1} e^{-c_2 x^{1/(1-\ga)}},
\end{equation}
where the positive constants $C_2, c_1, c_2$ can be expressed explicitly in
$\nu$, $q$ and $r(1),\dots,r(q)$.
This density $f_W$ is of a type known as $H$-function, see
\cite[Addendum]{SJ244}. 
\end{remark}

\section{Small $r(m)$}\label{Ssmall}

In this section we develop the results in the case of small $r(m)$,
i.e., $\sum_m r(m)\qww=\infty$.
However, we begin with some results holding in general, although their main
interest is for the case of small $r(m)$.

Let, recalling \eqref{eq:ly},
\begin{equation}\label{gss}
  \gss_m:=\Var Y_m = \sumkm \frac{\gd_k}{r(k)^2}
\end{equation}
and, as a simpler approximation,
\begin{equation}\label{hgss}
  \hgss_m:= \sumkm \frac{1}{r(k)^2}.
\end{equation}
We have, by the argument in \eqref{eq:yk},
\begin{equation}\label{gss2}
  \gss_m:=\hgss_m-\sum_{\ell=2}^{r(m)} \frac{1}{\ell^2}
=\hgss_m-O(1).
\end{equation}
Hence, the condition that $r(m)$ is small can be written in three equivalent
forms: 
\begin{equation}\label{gssoo}
\sum_m r(m)\qww=\infty 
\iff
  \hgss_m\to\infty
\iff
\gss_m\to\infty
\end{equation}
Furthermore, when this holds, then $\gs_m\sim \hgs_m$.

For convenience, define 
\begin{align}
  A_m&:=\sumkom e^{Y_k}E_k',  \label{A}
\\
B_m&:=\sumkom e^{Y_k}. \label{B}
\end{align}
Hence, \refT{TN2} says that \as{} $N_m\sim A_{m-1}$ as \ntoo.

\begin{lemma}
  \label{LC1}
For any sequence $r(m)$,
  \begin{equation}
    B_m = r(m)e^{Y_m+\Op(1)}, \label{lc1}
\qquad m\ge1
,
  \end{equation}
i.e., 
\begin{equation}
  \log B_m -\log r(m) - Y_m = \Op(1).
\end{equation}
\end{lemma}

\begin{proof}
  Let $m\ge1$ and let $m_1:=m-\ceil{r(m)/2}\ge0$.
Then, by \eqref{B},
\begin{align}\label{skrubb}
B_m e^{-Y_m} \ge \sum_{k=m_1}^m e^{Y_k-Y_m} 
\ge (m-m_1) e^{Y_{m_1}-Y_m } 
\ge \frac{r(m)}2 e^{Y_{m_1}-Y_m }. 
\end{align}

If $i<m-m_1=\ceil{r(m)/2}$, then $i\le \floor{r(m)/2}$ and thus
\begin{equation}\label{kank}
r(m-i)\ge r(m)-i \ge r(m)-\floor{r(m)/2}=\ceil{r(m)/2}.  
\end{equation}
Hence,
\begin{align}
\E |Y_m-Y_{m_1} | 
&=
\E (Y_m-Y_{m_1})
= y_m - y_{m_1}
=
\sum_{k=m_1+1}^m\frac{\gd_k}{r(k)}
\notag\\&
=\sum_{i=0}^{m-m_1-1}\frac{\gd_{m-i}}{r(m-i)}
%\notag\\&
\le (m-m_1)\cdot\frac{1}{\ceil{r(m)/2}}
=1.
\end{align}
Thus
$Y_m-Y_{m_1}=\Op(1)$, which by \eqref{skrubb} yields the lower bound
\begin{equation}\label{BYlower}
  B_me^{-Y_m}\ge r(m) e^{-(Y_m-Y_{m_1})-\log 2} = r(m) e^{\Op(1)}.
\end{equation}

In the other direction, we note that $\gd_k=0$ for exactly $r(m)-1$ values
of $k\le m$.
Hence, if $r(m)\le j\le m$, then $\gd_k=1$ for more that $j-r(m)$ values of
$k\in[m-j+1,m]$. For each such $k$,
\begin{equation}
\E e^{-E_k/r(k)}=(1+1/r(k))\qw\le (1+1/r(m))\qw.
\end{equation}
Consequently, by \eqref{eq:ly},
\begin{equation}
  \E e^{Y_{m-j}-Y_m}
= \prod_{\substack{m-j<k\le m\\\gd_k=1}} \E e^{-E_k/r(k)}
\le \Bigparfrac{1}{1+1/r(m)}^{j-r(m)}
.%= \bigpar{1+1/r(m)}^{r(m)-j}
\end{equation}
In other words, if $r(m)\le j\le m$, then 
\begin{equation}
  \E e^{Y_{m-j}-Y_m}
\le  \bigpar{1+1/r(m)}^{r(m)-j}
\le e \bigpar{1+1/r(m)}^{-j},
\end{equation}
and the same estimate holds trivially if $0\le j<r(m)$ too.
Consequently,
\begin{equation}
  \E \bigpar{B_me^{-Y_m}} 
%\le \sum_{j=0}^m e \bigpar{1+1/r(m)}^{-j}
\le \sum_{j=0}^\infty e \bigpar{1+1/r(m)}^{-j}
= e(r(m)+1)
\le 2e r(m)
\end{equation}
and thus, noting that $\log_+\Op(1)=\Op(1)$,
\begin{equation}\label{BYupper}
 B_me^{-Y_m} = r(m) \Op(1) \le r(m) e^{\Op(1)}. 
\end{equation}
 The result follows by combining \eqref{BYlower} and \eqref{BYupper}.
\end{proof}

\begin{lemma}\label{LC2}
For any sequence $r(m)$,
\begin{equation}
  \label{lc2}
  A_m=B_m e^{\Op(1)} = r(m) e^{Y_m+\Op(1)}.
\end{equation}
%If $r(m)\to\infty$ as \mtoo, then \XXX
\end{lemma}
\begin{proof}
  First condition on the entire sequence $\YYY$. Then, by \eqref{A}--\eqref{B},
  \begin{align}
\E\Bigpar{(A_m-B_m)^2\mid\YYY}
&=\E\Bigpar{\Bigpar{\sumkom e^{Y_k}(E_k'-1)}^2\Bigm|\YYY} 
\notag\\&
=\sumkom e^{2Y_k}
\le \sumkom e^{Y_k+Y_m}
=B_m e^{Y_m}
  \end{align}
and thus
  \begin{align}
\E\Bigpar{\Bigpar{\frac{A_m}{B_m}-1}^2\Bigm|\YYY} 
%\notag\\&
\le 
B_m\qw e^{Y_m}.
\label{gulla}
  \end{align}
Using \eqref{skrubb}, with $m_1:=m-\ceil{r(m)/2}$ as above,
and taking the expectation, we obtain
  \begin{align}
\E\Bigpar{\frac{A_m}{B_m}-1}^2
\le 
\E\bigpar{B_m\qw e^{Y_m}}
\le \frac{2}{r(m)}\E e^{Y_m-Y_{m_1}}.
\label{kalla}
  \end{align}
Furthermore, by \eqref{eq:ly} and \eqref{kank}, if $r(m)\ge3$, 
\begin{align}
  \E e^{Y_m-Y_{m_1}} 
&\le \prod_{k=m_1+1}^m \E e^{E'_k/r(k)}
\le 
\Bigpar{1-\frac{1}{r(m_1+1)}}^{-(m-m_1)}
\notag\\&
\le 
\Bigpar{1-\frac{1}{\ceil{r(m)/2}}}^{-\ceil{r(m)/2}}
\le 4.
\label{kalle}
\end{align}
Consequently, if $r(m)\to\infty$ as \mtoo, then
\eqref{kalla} and \eqref{kalle} imply $\E\bigpar{A_m/B_m-1}^2\to0$, and thus
$A_m/B_m\pto1$, which is the same as $\log(A_m/B_m)=\op(1)$, or
$A_m=B_m e^{\op(1)}$. In particular, $A_m=B_me^{\Op(1)}$, and \eqref{lc2}
follows in this case by \refL{LC1}.

On the other hand, if $\sup_m r(m)<\infty$, we note that the \rhs{} of
\eqref{gulla} is $\le1$, and thus taking the expectation yields
  \begin{align}
\E\Bigpar{\frac{A_m}{B_m}-1}^2
\le 
\E\bigpar{B_m\qw e^{Y_m}}
\le 1,
\label{kulla}
  \end{align}
which implies that $A_m/B_m=\Op(1)$ and thus, using \eqref{lc1},
\begin{equation}\label{pik}
A_m\le B_me^{\Op(1)} = r(m)e^{Y_m+\Op(1)}.  
\end{equation}
Furthermore, assuming $r(m)=O(1)$ and thus $\log r(m)=O(1)$,
\begin{equation}\label{staaf}
  A_m\ge e^{Y_m} E'_m = e^{Y_m + \Op(1)}= r(m)e^{Y_m + \Op(1)}.
\end{equation}
The result in this case (bounded $r(m)$)
follows from \eqref{pik} and \eqref{staaf}, together with \eqref{lc1} again.
\end{proof}

\begin{theorem}\label{TNO}
For any sequence $r(m)$,
\begin{equation}\label{tno}
  N_m = r(m) e^{Y_m+\Op(1)}.
\end{equation}
\end{theorem}
\begin{proof}
  \refT{TN2} yields $N_m/A_{m-1}\asto1$, and thus
$N_m/A_{m-1}\pto1$, which is the same as $\log(N_m/A_{m-1})=\op(1)$.
Hence, \refL{LC2} yields
\begin{align}\label{tnq}
  \log N_m = \log A_{m-1} + \op(1) = \log r(m-1) + Y_{m-1} + \Op(1).
\end{align}
Furthermore, $\log r(m-1)=\log r(m)+O(1)$ and 
$Y_m-Y_{m-1}=\frac{\gd_m}{ r(m)} E_m'=\Op(1)$. Hence, \eqref{tno} follows from
\eqref{tnq}. 
\end{proof}

\refT{TNO} holds for any sequence $r(m)$, but if $r(m)$ is large, then
\refT{TN4} gives a sharper result.
In the remainder of this secion we assume that 
$r(m)$ is small.

\begin{lemma}
  \label{LD}
Suppose that $\sum_m r(m)\qww=\infty$.
Then, as \mtoo,
\begin{equation}\label{ld}
  \frac{Y_m-y_m}{\gs_m}\dto N(0,1).
\end{equation}
\end{lemma}

\begin{proof}
An immediate consequence of \eqref{Yy} and the central limit theorem with
Lyapounov's condition \cite[Theorem 7.2.2]{Gut},   
using the estimate
\begin{equation}
  \sumkm \E \Bigabs{\frac{\gd_k}{r(k)}(E_k-1)}^3
%=   \sumkm \frac{\gd_k}{r(k)^3}\E|E_k-1|^3
=
 C \sumkm \frac{\gd_k}{r(k)^3}
\le C \gs_m^2
=o(\gs_m^3)
.
\end{equation}
\end{proof}

\begin{theorem}\label{TNsmall}
Suppose that $\sum_m r(m)\qww=\infty$.  
Then
\begin{equation}
  \label{tnsmall}
  \frac{\log N_m -\bigpar{y_m+\log r(m)}}{\hgs_m} \dto N(0,1).
\end{equation}
\end{theorem}

\begin{proof} 
Since $\hgs_m\to\infty$ 
by the assumption and \eqref{gssoo},  
\refT{TNO} yields
  \begin{align*}
    \log N_m - \bigpar{y_m+\log r(m)} = Y_m + \Op(1)-y_m
= Y_m-y_m+\op(\hgs_m)
  \end{align*}
and thus
  \begin{align}
\frac{\log N_m - \bigpar{y_m+\log r(m)}}{\hgs_m} 
= \frac{\gs_m}{\hgs_m}\cdot\frac{Y_m-y_m}{\gs_m}+\op(1).
  \end{align}
Since \eqref{gss2} implies $\gs_m/\hgs_m\to1$,  
the result follows from \refL{LD} and the standard Cram\'er--Slutsky theorem
\cite[Theorem 5.11.4]{Gut}.
\end{proof}

The asymptotic distribution of $N_m$ is thus log-normal, for any small
sequence $r(m)$.
Under weak regularity assumptions, this can be inverted to yield asymptotic
normality of $M_n$.
For convenience, we will assume that the sequence $r(m)$ is regularly
varying, see \eg{} \cite[\S1.9]{RegVar}.
We define, 
\begin{align}\label{hy}
  \hy_m:=\sumkm \frac{1}{r(k)},
\end{align}
and see by  \eqref{eq:yk} that
\begin{align}\label{hyy}
  \hy_m-y_m = \sum_{\ell=2}^{r(m)}\frac{1}{\ell} = \log r(m)+O(1).
\end{align}

\begin{lemma}\label{LR}
Suppose that $\sum_m r(m)\qww=\infty$, and that $r(m)$ is regularly varying.
\begin{romenumerate}
\item \label{LRgb}  
Let $\gb(m):=r(m)\hgs_m$ as in \eqref{ti2gb}.
Then, for all $m\ge1$,
\begin{align}
\gb(m)&\ge  m\qq, \label{lrgb-}
\\
\gb(m)&\le C m^{0.51}=o(m).    \label{lrgb+}
\end{align}

\item \label{LRxm}
Suppose further that $m=m_n$ and $\xm=m_n$ are two sequences such that, as
\ntoo,
$\xm\sim m$. 
Then, 
\begin{align}
  r(\xm)&\sim r(m), \label{lrr}
\\
\hgs_{\xm}&\sim\hgs_m.\label{lrgs}
\end{align}
\end{romenumerate}
\end{lemma}

\begin{proof}
Suppose that $r(m)$ is regularly varying of index $\kk$.
Then $r(m)=m^{\kk+o(1)}$, as a consequence of 
\cite[Theorem 1.9.7 or Theorem 1.5.4]{RegVar}.
Hence, the assumption $\sum_m r(m)\qww=\infty$ implies $\kk\le0.5$. 
Hence, $r(m)=O\bigpar{ m^{0.5+o(1)}}$, and in particular, 
\begin{align}\label{0.51}
r(m)=O\bigpar{ m^{0.51}} = o(m).  
\end{align}

Since $\kk<0.51$, we have $r(m)/r(k)\le C (m/k)^{0.51}$
whenever $1\le k\le m$;
this is another consequence of \cite[Theorem 1.9.7, or Theorem 1.5.6]{RegVar}.
Hence,
\begin{align}
  \gb(m)^2=\sumkm\frac{r(m)^2}{r(k)^2}
\le \sumkm C \frac{m^{1.02}}{k^{1.02}} \le C m^{1.02},
\end{align}
which yields \eqref{lrgb+}.
%and thus $\gb(m)=O\bigpar{m^{0.51}}=o(m)$.
In the opposite direction, trivially
\begin{align}
  \gb(m)^2=\sumkm\frac{r(m)^2}{r(k)^2}
\ge \sumkm 1=m.
\end{align}

\pfitemref{LRxm}
Since $r(m)$ is regularly varying, 
if $\xm\sim m$ then $r(\xm)\sim r(m)$, as a consequence of 
\cite[Theorem 1.9.7 or Theorem 1.5.2]{RegVar}.

Furthermore, $r(m)\qww$ is regularly varying of index $-2\kk$, and it
follows, using the assumption $\sum_m r(m)\qww=\infty$ when $\kk=1/2$, 
that $\hgss_m$ defined by \eqref{hgss} is regularly varying of index
$1-2\kk$, see 
\cite[Theorem 1.9.5(ii) and Propositions 1.5.8 and  1.5.9b]{RegVar}.
Hence $\hgss_{\xm}\sim\hgss_m$.
%, and thus $\gs_{\xm}\sim\hgs_{\xm}\sim\hgs_{m}\sim\gs_{m}$.
\end{proof}

\begin{proof}[Proof of \refT{TI2}]
Note first that \eqref{ti2mu} implies that $\mu(n)\to\infty$ as \ntoo. We
may  thus assume $\mu(n)\ge1$. Furthermore, 
\eqref{lrgb-} yields $\gam(n)=\gb(\floor{\mu(n)})\to\infty$ as \ntoo.
We may thus also replace $\mu(n)$ by $\floor{\mu(n)}$, and thus in the
sequel assume that $\mu(n)$ is an integer.

Fix $x\in\bbR$ and let
$m:=\ceil{\mu(n)+x\gam(n)}$.
By \eqref{ti2gam} and   \eqref{lrgb+}, as \ntoo,
\begin{align}
  \gam(n)=\gb\bigpar{\mu(n)}=o\bigpar{\mu(n)}
\end{align}
and thus $m\sim \mu(n)$. In particular, $m\to\infty$ as \ntoo{} and 
$m\ge1$ for all large $n$;  consider only such $n$.
Then, using \eqref{eq:NM},
\begin{align}
&  \P\bigpar{M_n\ge \mu(n)+x\gam(n)}
= \P\bigpar{M_n\ge m}
= \P\bigpar{N_m\le n}
\notag
\\&\qquad
= \P\lrpar{\frac{\log N_m - (y_m+\log r(m))}{\hgs_m} 
\le \frac{\log n - (y_m+\log r(m))}{\hgs_m}}.\label{julie}
\end{align}
In the last line, the random variable on the left of '$\le$' is
asymptotically normal by \eqref{tnsmall}; we turn to the term on the right.
By \eqref{ti2mu} and \eqref{hy}, $\hy_{\mu(n)}=\log n+O(1)$, and by
\eqref{hyy} this yields
\begin{align}\label{cec}
  \log n - \bigpar{y_m+\log r(m)}
= \hy_{\mu(n)} -\hy_m + O(1).
\end{align}
Write $\xm:=\mu(n)$.
Since $\xm\sim m$, it follows by \eqref{lrr} that
$r(k)\sim r(m)$ uniformly for all $k$ with $\xm\le k\le m$ or $m\le k\le
\xm$.
Consequently, by \eqref{eq:yk},
\begin{align}\label{zez}
  \hy_{\xm}-\hy_m \sim \frac{\xm-m}{r(m)} 
.%= \frac{-x\gam(n)+O(1)}{r(m)}
%=\frac{-x r(\xm)\hgs_{\xm}}{r(m)}+O(1)
\end{align}
Furthermore, by \eqref{lrr}--\eqref{lrgs},
\begin{align}
\gam(n)=\gb(\xm)=r(\xm)\hgs_{\xm} \sim r(m)\hgs_m,
\end{align}
and hence, using \eqref{lrr}--\eqref{lrgs} again,
\begin{align}
\frac{\xm-m}{r(m)} 
& \nonumber
= \frac{-x\gam(n)+O(1)}{r(m)}
=\frac{-x r(\xm)\hgs_{\xm}}{r(m)}+O(1)
\\&
=-x\hgs_m\etto+O(1).\label{xex}
\end{align}
Since $\hgs_m\to\infty$ by \eqref{gssoo}, it follows from \eqref{xex}
that
\begin{align}
\frac{\xm-m}{r(m)\hgs_m} 
\to -x
\end{align}
and thus \eqref{cec} and \eqref{zez} yield
\begin{align}\label{cac}
\frac{  \log n - \bigpar{y_m+\log r(m)}}{\hgs_m}
= \frac{\hy_{\xm} -\hy_m}{\hgs_m} + o(1)
\to -x.
\end{align}
Consequently \eqref{julie} and \refT{TNsmall} imply, with $\zeta\xsim N(0,1)$,
\begin{align}\label{thoth}
  \P\bigpar{M_n\ge\mu(n)+x\gam(n)} \to \P(\zeta \le -x) = \P(\zeta\ge x)
\end{align}
for every $x\in\bbR$, which proves \eqref{ti2}.

Finally, the limit \eqref{ld} is evidently mixing, i.e., it holds also
conditioned on any fixed set $Y_1,\dots,Y_{K}$ of the variables.
(See \cite[Proposition 2]{mixing}.) 
Using the argument in the proof of \refL{L2}, \eqref{ld} holds also 
conditioned on the sequence of indicators $\indic{X_k}$ is accepted, $k\le
N_K$, which is equivalent to conditioning on $M_1,\dots,M_{N_K}$.
It follows that the results above, including
 \eqref{thoth} and \eqref{ti2}, hold also conditioned on 
$M_1,\dots,M_{N_K}$,
and thus a fortiori also conditioned on $M_1,\dots,M_K$.
Hence \eqref{ti2} is mixing.
\end{proof}

\begin{remark}
  It can be seen from the proof that the assumption that $r(m)$ be regularly
  varying may be replaced by the weaker
  \begin{align}
    \text{if}\quad\xm=m+O(\gb(m)), 
\quad \text{then}\quad
r(\xm)\sim r(m).
  \end{align}
\end{remark}

\section{Examples}\label{Sex}

\begin{example} \label{Emedian} 
  Consider \hmedian. By \eqref{hmedian}, this is the case
$r(m)=\floor{m/2}+1$.
Hence, \eqref{per} holds with $\nu=1$ and $q=2$, and thus
\refT{Tper} applies and yields
\begin{align}\label{EWmed}
  \E W^s = 2^s \gG\Bigpar{\frac{s}2+1},
\end{align}
which shows that $(W/2)^2\xsim\Exp(1)$ and that $W$ thus has a Rayleigh
distribution with density function
\begin{equation}\label{fWmed}
  \tfrac12 x e^{-x^2/4}, 
\qquad x>0.
\end{equation}
Consequently, \refT{TI1} shows that $M_n/n\qq$ converges 
in distribution to this Rayleigh distribution, as shown by
\citet{HelmiP-median2013}; 
moreover, \refT{TI1} shows \as{} convergence, and convergence of all moments.
(\cite{HelmiP-median2013} treated only the mean.)

Furthermore,
the definition \eqref{gd} yields $\gd_k=\indic{m\text{ odd}}$, and
consequently, \eqref{Z} yields
\begin{align}
  Z=\sumj \frac{1}{r(2j-1)}(E_{2j-1}-1)
=\sumj \frac{1}{j}(E_{2j-1}-1)
\eqd\sumj \frac{1}{j}(E_{j}-1).
\end{align}
It is well-known that this sum yields a centered Gumbel distribution:
\begin{equation}\label{Zgumbel}
  \P(Z\le x) = e^{-e^{-(x-\gam)}}.
\end{equation}
(Recall that by definition, $\E Z=0$.)
This can be seen in several ways, for example by computing the \mgf{} 
$\E e^{sZ}=e^{-\gam s}\gG(1-s)$, $\Re s<1$, by arguments similar to the proof
of \refT{Tper}, or directly from \eqref{W} and the identification of $W$ as
a Rayleigh distribution; we omit the details.
\end{example}

\begin{example}  \label{Epercentile}
Let $0<\ga\le1$ and consider \hpercentile{\ga}, \ie{} $r(m)=\ceil{\ga m}$ by
\eqref{hpercentile}. 
\refT{TI1} applies and shows convergence $M_n/n^\ga\to W_\ga$, a.s., in
distribution, and with all moments, to some positive random variable $W_\ga$
with moments given by \eqref{ti1}.
In particular, when $\ga$ is rational, the moments can be calculated (in
terms of the Gamma function) by \refT{Tper}.
We give a few examples in \refTab{TABpercent}.
(The expectations were given in  \cite{GaitherWard2012}, see below.
Note that the results can be written in different forms, using standard
  Gamma  functions identities; 
% for $\ga=2/5$ we have also simplified using the duplication  formula. 
cf.\ the partly different but equivalent formulas for the expectations 
given here and in  \cite{GaitherWard2012}.)
In particular, note that \eqref{tomas4} yields for the special case
$\nu=1$ (where $r(i)=1$ for $1\le i\le q$)
\begin{align}\label{EW1}
\E W_{1/q}^s &= \frac{q^{s+1}}{s+q} \gG\Bigpar{\frac{s}{q}+1}^{q-1} 
= \frac{q^{s+2-q}s^{q-1}}{s+q} \gG\Bigpar{\frac{s}{q}}^{q-1} 
\end{align}
and for $\nu=q-1$
(where $r(i)=i$ for $1\le i\le q-1$ and $r(q)=\nu=q-1$)
\begin{align}\label{EWq-1}
\E W_{\nuq/q}^s &
= \frac{q^{s+1}}{\nuq^{\nuq s/q}\xpar{\nuq s+q}} 
  \gG\Bigpar{\frac{s}{q}+1}
= \frac{sq^{s}}{\nuq^{\nuq s/q}\xpar{\nuq s+q}} 
  \gG\Bigpar{\frac{s}{q}}.
\end{align}

\begin{figure}
  \centering
  \begin{tabular}{r|c|c}
    $\ga$ &  $\E W_\ga$ & $\E W_\ga^2$ \\
\hline
1 & $\frac12$ & $\frac13$\\[3pt]
$\frac12$ & $\frac{2\sqrt\pi}{3}$ & $2$ \\[3pt]
$\frac13$ & $\frac{1}{4}\gG\bigparfrac{1}{3}^2$ 
   & $\frac{12}{5}\,\gG\bigparfrac{2}{3}^2 $ \\[3pt]
$\frac23$ & $\frac{3}{2^{2/3}\cdot5}\gG\bigparfrac{1}{3}$ 
   & $\frac{9}{2^{1/3}\cdot7}\gG\bigparfrac{2}{3}$  \\[3pt]
$\frac14$ & $\frac{1}{20}\gG\bigparfrac{1}{4}^3$
  & $\frac{4}{3}\pi^{3/2}$ \\[3pt]
$\frac34$ & $\frac{4}{3^{3/4}\cdot7}\gG\bigparfrac{1}{4}$ 
  & $\frac{16}{3^{3/2}\cdot5}\pi\qq$  \\[3pt]
$\frac15$ & $\frac{1}{150}\gG\bigparfrac{1}{5}^4$
  & $\frac{16}{35}\gG\bigparfrac{2}{5}^4$ \\[3pt]
$\frac25$ & $\frac{2^{1/5}}{7}\gG\bigparfrac{1}{5}\gG\bigparfrac{2}{5}$
 & $\frac{2^{7/5}\cdot 5}{9}\gG\bigparfrac{2}{5}\gG\bigparfrac{4}{5}$ \\[3pt]
%%SJ251a:=GAMMA(1/5)^2*GAMMA(7/10)/(2^(2/5)*7*sqrt(Pi));
%SJ251b:=GAMMA(1/5)*GAMMA(2/5)*2^(1/5)/7;
%SJ252:=GAMMA(2/5)*GAMMA(4/5)*2^(7/5)*5/9;
$\frac35$ & $\frac{5\gG\parfrac{13}{15}\gG\parfrac{1}{5}}
  {{3^{3/5}\cdot 8}\gG\parfrac23}$
 & $\frac{10\gG\parfrac{1}{15}\gG\parfrac{2}{5}}
       {{3^{11/5}\cdot 11}\gG\parfrac23}$ 
%SJ351:=GAMMA(13/15)*GAMMA(1/5)*5/(3^(3/5)*8*GAMMA(2/3));
%SJ352:=GAMMA(1/15)*GAMMA(2/5)*10/(3^(11/5)*11*GAMMA(2/3));
\\[6pt]
$\frac45$ & $\frac{5}{2^{8/5}\cdot9}\gG\bigparfrac{1}{5}$
 & $\frac{25}{2^{11/5}\cdot13}\gG\bigparfrac{2}{5}$
\\[3pt]
%SJ451:=GAMMA(1/5)*5/(9*2^(8/5));
%SJ452:=GAMMA(2/5)*25/(13*2^(11/5));
\hline
  \end{tabular}
  \caption{Some values of the asymptotic first and second moments for  
\hpercentile{\ga}.}
  \label{TABpercent}
\end{figure}

The  expectations $c_\ga:=\E W_\ga = \lim_{\ntoo} \E M_n/n^\ga$ have been
considered before.
\citet{KriegerPS-rank2007} found that $W_1\sim U(0,1)$ and thus 
$c_1=1/2$, but otherwise showed only existence of the limit $c_\ga$.
\citet{GaitherWard2012} computed 
$c_\ga$ (our $\E W_\ga$)
as
\begin{equation}\label{GW}
\newcommand\pga{\ga}
c_\ga= \frac{1+\sum_{k\ge1}
  \frac{\ceil{\pga k}-\pga k}{\ceil{\pga k}}\prod_{j=1}^k \frac{1}{1+\xfrac{\pga }{\ceil{\pga j}}}}
{(\pga +1)\gG(\pga +1)}.
\end{equation}
In the case $\ga=\nu/q$ rational, 
they showed further how this can be transformed
into a form that they could evaluate symbolically; as examples they gave
explicit values for all cases with $q\le 6$.
The formula \eqref{GW} must agree with \eqref{ti1} for $s=1$, i.e.,
\begin{equation}\label{cga}
 c_\ga=\E W_\ga= \frac{1}{\gG(\ga+2)}\prodk\frac{1+1/k}{1+\ga/\ceil{\ga k}},
\end{equation}
but we do not see any direct proof of this.
The explicit values for $\ga$ rational are obtained more easily from
\eqref{tomas}--\eqref{tomas4}; 
in particular, for the cases $\nu=1$ and $\nu=q-1$, we can take $s=1$ in
\eqref{EW1}--\eqref{EWq-1}.

\citet{GaitherWard2012} gave a graph of the function $\ga\mapsto c_{\ga}$,
and conjectured that it is continuous at all irrational $\ga$ but only
left-continuous at rational $\ga$.
This is easily verified from our form \eqref{cga}, since 
the infinite product converges unifomly on each interval $[a,1]$, 
and each factor in it is continuous at irrational $\ga$ and left-continuous
everywhere, but for each rational $\ga$ there are factors that have jumps;
furthermore, the jumps in the factors are always positive. Hence, $c_\ga$ has
a positive jump at each rational $\ga\in(0,1)$.

Let us consider the case $\ga=1/2$ in more detail. 
For comparisons, let $\Wmed$ denote the limit variable for \hmedian{}
in \refE{Emedian}.
Then \eqref{tomas2} and \eqref{EWmed} yield
\begin{align}
\E W_{1/2}^s = \frac{2^s}{1+s/2} \gG\Bigpar{\frac{s}2+1}  
=\frac{1}{s/2+1}\E \Wmed^s
\end{align}
and thus 
\begin{align}
  \label{WW1/2}
W_{1/2}\eqd U\qq\Wmed, 
\end{align}
where $U\sim U(0,1)$ is independent of $\Wmed$.
In particular, $\E W_{1/2}=2\sqrt\pi/3=\frac{2}3\E\Wmed$, as shown by
\citet{HelmiP-median2013}, who also gave a formula for the density function
of $W_{1/2}$ which is equivalent to \eqref{WW1/2}.

The relation between $W_{1/2}$ and $\Wmed$ is studied further in \refE{E+1}
and \refE{Er-}.
\end{example}

\begin{example}\label{E+1}
One way to regard the difference between \hmedian{} and \hpercentile{\frac12}
is that their sequences $(r(m))_0^\infty$ are $1,1,2,2,\dots$ and
$1,1,1,2,2,\dots$, respectively, and thus differ only by an extra 1 in the
latter case.

We can study this in general. Given a sequence $r(m)$ satisfying
\eqref{eq:rr}, define a new sequence $\xr(m)$ by inserting an extra 1 first,
i.e., let $\xr(m):=r(m-1)$, $m\ge1$. We use $\tilde{}$ to denote variables
for the new sequence.
It follows from \eqref{Z} that
\begin{align}\label{xZ}
  \xZ\eqd Z+E-1
\end{align}
with $E\xsim\Exp(1)$ independent of $Z$. Suppose now that $r(m)=\ga m+O(1)$,
or more generally that \eqref{sofie} holds. Then the same is true for
$\xr(m)$; furthermore it is easy to see from \eqref{rho} that
$\xrho=\rho+1$, and thus
\eqref{W} yields, using \eqref{xZ},
\begin{align}
  \xW \eqd e^{-\ga E} W = U^{\ga} W
\end{align}
where $U=e^{-E}\xsim U(0,1)$ is independent of $W$.
Equivalently, 
\begin{align}\label{erika}
\E \xW^s = \frac{1}{1+\ga s}\E W^s,
\end{align}
which also follows from \eqref{ti1}.

In fact, this has a simple probabilistic explanation. 
In the modified strategy, the first candidate is, as always, accepted, and
because $\xr(1)=1$, the threshold for the next candidate is $\xY_1=X_1$. 
Since the threshold never decreases (see \refL{L2}), this means that only
candidates better than $X_1$ have a chance of being considered. Moreover, it
is easy to see that if we consider only the subsequence of candidates with
values $X_n>X_1$, then the ones hired by the modified strategy are precisely
those that would have been hired by the original strategy applied to this
subsequence of candidates. Conditioning on $X_1=x_1$, the values in the
subsequence will be independent with the conditional distribution 
$\cL(E+x_1)$, and subtracting $x_1$ from all values, we obtain the original
problem for the original sequence. However, still conditioned on $X_1$,
if we start with a sequence of
$n$ candidates, the subsequence will contain only $\Bin(n-1,e^{-X_1})$
candidates. Note that $U:=e^{-X_1}\xsim U(0,1)$.
It follows, using the law of large numbers,
that if \refT{TI1} applies to the original strategy, then
it holds for the modified one too, with
\begin{align}
\tW \eqd U^\ga W,
\end{align}
where $U\xsim U(0,1)$ is independent of $W$.
\end{example}

\begin{example}\label{Er-}
  Another way to view the difference between \hmedian{} and
  \hpercentile{\frac12} is that $r(m)$ has been decreased by 1 for every
  even $m\ge2$. Let us consider, in general, 
the effect of decreasing a single value
  $r(m)$ by 1, assuming that this is possible (i.e., that $r(m-1)<r(m)=r(m+1)$).
Assume also for simplicity that \refT{TI1} applies.
Then \eqref{ti1} shows that $W$ is modified such that $\E W^s$ is multiplied
by
\begin{equation}
  \frac{1+s\ga/r(m)}{1+s\ga/(r(m)-1)} 
= \frac{r(m)-1}{r(m)} +   \frac{1/r(m)}{1+s\ga/(r(m)-1)} 
= \E V^{s\ga/(r(m)-1)},
\end{equation}
where $V$ has density $1/r(m)$ on $(0,1)$ and a point mass
$\P(V=1)=1-1/r(m)$.
Hence, the modified limit $\xW\eqd V^{\ga/(r(m)-1)}W$.
This can be repeated for several changes.

In particular, looking just at the expectation, decreasing $r(2)$ from 2 to
1 in \hmedian{} multiplies $\E W$ by $(1+\frac{1}4)/(1+\frac12)=5/6$.
As seen above, $\E W$ decreases by a factor $2/3$ if we change \hmedian{} to
\hpercentile{1/2}, and we now see that half of the decrease is due to the
decrease of $r(2)$.
This illustrates that, as said in \refS{S:intro}, 
in the case of large $r(m)$, the asymptotic behaviour
is heavily influenced by the effects of the first candidates.
\end{example}

\begin{example}\label{Ekapad}
Another variation of \hmedian{} is to take the sequence 1,2,2,3,3,\dots,
i.e.,
\begin{align}
  r(m)=\ceil{m/2}+1,
\qquad m\ge0.
\end{align}
\refT{Tper} applies and yields, \eg{} by \eqref{tomas2},
\begin{align}\label{eleonora}
  \E W^s = 2^s\gG\Bigpar{\frac{s}2+2}.
\end{align}
This follows also from \refE{E+1}, since if we insert an extra 1 first in
this sequence $r(m)$, we obtain \hmedian{} as in \refE{Emedian}, 
and thus by \eqref{erika} and \eqref{EWmed},
\begin{align}
  2^s\gG\Bigpar{\frac{s}2+1} = \frac{1}{s/2+1}\E W^s.
\end{align}
It follows from \eqref{eleonora} that $W^2/4\xsim\gG(2)$, and thus $W$ has
the density
\begin{align}\label{Wkapad}
  f_W(x) = \frac{x^3}{8}e^{-x^2/4}, \qquad x>0.
\end{align}
Equivalently, $W/\sqrt2\sim\chi(4)$, a chi distribution.

We return to the significance of this example in \refS{Scond}.
\end{example}

\begin{example} \label{Erecord}
  The extreme case of small $r(m)$ is $r(m)=1$, $m\ge0$.
This means the we only accept candidates that are better than all previous
candidates, i.e., the record values in the sequence $(X_n)$. 

\refT{TI2} applies with $\mu(n)=\log n$, $\gb(m)^2=m$ and
$\gam(n)^2=\floor{\log n}\sim \log n$, which yields
\begin{align}
\frac{  M_n-\log n}{\log\qq n}\dto N(0,1).
\end{align}

This is a well-known result for the number of records, see \eg{}
\cite[Theorem 7.4.2]{Gut}, and is easily proved directly by the central
limit theorem, 
observing that the indicators $I_k:=\indic{X_k \text{ is a record}}$
are independent with $I_k\sim \Be(1/k)$.
See further the next example.
(This connection between records and the hiring problem was noted by 
\cite{BroderEtAl}.)
\end{example}

\begin{example}\label{Ebest}
More generally, consider \hbest{r} for a fixed $r\ge1$, with $r(m)$ given by
\eqref{hbest}. Thus \refE{Erecord} is the case $r=1$.
This strategy was studied by \citet{ArchibaldMartinez2009} and, in great
detail, by
\citet{HelmiMP-best2014}.
A value $X_k$ is accepted if it is an \emph{$r$-record}, in the sense that it is
one of the $r$ best values seen so far.
(In particular, 
the first $r$ values $X_k$ are always accepted.)

\refT{TI2} applies with $\mu(n)= r\log n$, $\gb(m)^2\sim m$ and
$\gam(n)^2\sim r\log n$, which yields
\begin{align}
\frac{  M_n-r\log n}{(r\log n)\qq}\dto N(0,1),
\end{align}
as shown by \cite{HelmiMP-best2014} (who also gave many other results,
including for fixed $n$, and for the case when both $n,r\to\infty$).
Again, this is easily shown directly by the central limit theorem, using the
fact that 
the indicators $I_k:=\indic{X_k \text{ is an $r$-record}}$
are independent with $I_k\sim \Be(r/k)$ for $k\ge r$, which is noted in
\cite{HelmiMP-best2014}, see also the furthern references given there.
 \end{example}

 \begin{example}
   Let $r(m):=\floor{\sqrt{m}}$, $m\ge1$.
This is an example of small $r(m)$. (But rather large among the small ones.)
We have 
\begin{align*}
  \gss_m = \hgss_m+O(1)
=\sumkm \frac{1}{\floor{\sqrt k}^2}+ O(1)
=\sumkm \frac{1}{k}+ O(1)
=\log m+ O(1)
\end{align*}
and
\begin{align}
  \hy_m=\sumkm\frac{1}{\floor{\sqrt k}}
=\sum_{j=1}^{\floor{\sqrt m}}\frac{2j+1}{j} + O(1)
=2\sqrt m + \tfrac12\log m + O(1).
\end{align}
It follows that \refT{TI2} applies with 
\begin{align}
\mu(n)= \tfrac{1}{4}\log^2n -\tfrac12\log n\log\log n,
\end{align}
$\gb(m)^2=r(m)^2\hgss_m\sim m\log m$
and
\begin{align}
\gam(n)^2\sim \mu(n)\log \mu(n) \sim \tfrac12\log^2n \log\log n.  
\end{align}
Hence, \refT{TI2} yields
\begin{align}
  \frac{M_n - (\tfrac{1}{4}\log^2n -\tfrac12\log n\log\log n)}
{\log n (\log\log n)\qq} \dto N\Bigpar{0,\frac12}.
\end{align}
 \end{example}

 \begin{example}
   \label{Eirreg}
We give an example showing how \refT{TI2} can fail when the sequence $r(m)$
is too irregular.

Let $r(m)$ be such that $r(8^i)=2^i$, 
$r(m)$ increases linearly on each interval $[8^i,8^i+2^i]$, and is constant
between these intervals; i.e.,
\begin{align}
  r(m):=
  \begin{cases}
    2^i+m-8^i, & 8^i\le m\le 8^i+2^i,\quad i\ge0\\
    2^{i+1}, & 8^i+2^i\le m\le8^{i+1},\quad i\ge0.
  \end{cases}
\end{align}
Thus $r(m)\asymp m^{1/3}$ (meaning that $cm^{1/3} \le r(m) \le C m^{1/3}$
for some constants $c$ and $C$), and thus \eqref{hgss}--\eqref{gss2} yield 
$\gss_m\sim\hgss_m\asymp m^{1/3}$. Hence $\gb(m)^2=r(m)^2\hgss_m\asymp m$ so
$\gb(m)\asymp m\qq$.
We see also that $\xm\sim m$ implies $\hgs_{\xm}\sim \hgs_m$.

Define $m_i=8^i$ and $n_i:=\floor{\exp\bigpar{\hy_{8^i}}}$, $i\ge1$.
Consider only the subsequence $(n_i)_i$. We have chosen $n_i$ such that
\eqref{ti2mu} holds with $n=n_i$ and $\mu(n)=8^i$, so we may choose $\mu(n)$
such that $\mu(n_i)=8^i=m_i$ for $i\ge1$.
Then $\gam(n_i)=\gb(8^i)\asymp 2^{3i/2}$.

Now argue as in the proof of \refT{TI2} for $n=n_i$. Thus
$\xm=\mu(n_i)=m_i=8^i$.
Suppose first that $x<0$. Then $m\le\xm$, and $r(k)=2^i=r(m_i)$ for every
$k\in[m,\xm]$ (at least if $i$ is large); hence, as $i\to\infty$,
\begin{align}
  \hy_{\xm}-\hy_m=\frac{\xm-m}{r(m_i)}\sim \frac{-x\gam(n_i)}{r(m_i)}
=\frac{-x\gb(m_i)}{r(m_i)}=-x\hgs_{m_i},
\end{align}
and since $\hgs_{m}\sim\hgs_{m_i}$, 
\begin{align}
\frac{\hy_{\xm}-\hy_m}{\hgs_m}\to-x.
\end{align}
On the other hand, if $x>0$, then $m\ge \xm$, and for most $k\in[\xm,m]$,
$r(k)=2^{i+1}=2r(m_i)$. Hence, as \itoo, similarly,
\begin{align}
  \hy_{\xm}-\hy_m\sim -\frac{m-\xm}{2r(m_i)}
\sim -\frac{x\gam(n_i)}{2r(m_i)}
%=\frac{-x\gb(m_i)}{r(m_i)}
=-\frac{x}2\hgs_{m_i},
\end{align}
and
\begin{align}
\frac{\hy_{\xm}-\hy_m}{\hgs_m}\to-\frac{x}2.
\end{align}
Instead of \eqref{thoth}, we thus obtain
\begin{align}
    \P\bigpar{M_{n_i}\ge\mu(n_i)+x\gam(n_i)} \to 
    \begin{cases}
\P(\zeta \le -x) = \P(\zeta\ge x), & x<0,
\\
\P(\zeta \le -x/2) = \P(2\zeta\ge x), & x>0.      
    \end{cases}
\end{align}
Consequently, along the sequence $(n_i)$, $M_n$ after normalization
as in \eqref{ti2} converges to the non-normal random variable
\begin{align}
\eta:=
    \begin{cases}
\zeta, & \zeta \le 0,
\\
2\zeta, & \zeta >0,
    \end{cases}
\end{align}
where $\zeta\in N(0,1)$.

On the other hand, for many other subsequences there is asymptotic
normality (by the same proof), for example for $\floor{\exp(\hy_{2\cdot 8^i})}$.
 \end{example}

\section{Conditioning on the first value}\label{Scond}
We have seen above that in the case of large $r(m)$, the asymptotics depend
heavily on the first values $X_k$, and thus in particular on the first value
$X_1$.  Furthermore, as have been remarked by \cite{BroderEtAl}, assuming
$r(1)=1$, so that the second accepted candidate is the first one with
$X_n>X_1$, the waiting time $N_2-N_1$ until the second candidate is accepted
has, conditioned on $X_1$, the distribution $\Ge\bigpar{e^{-X_1}}$ with
expectation 
\begin{align}
  \E \bigpar{N_2-N_1\mid X_1}= e^{X_1} =1/U,
\end{align}
where $U:=e^{-X_1}\xsim U(0,1)$. Consequently, $\E (N_2-N_1)=\E
U\qw=\infty$, and thus $\E N_m=\infty$ for every $m\ge2$.

These effects led \cite{BroderEtAl} to consider \hmedian{}
conditioned on $X_1$. We can do this in general. We assume that $r(m)$ is
large, since for small $r(m)$, conditioning on $X_1$ has no effect on the
asymptotics, see \eg{} the mixing property in \refT{TI2}.

\begin{theorem}\label{Tcheck}
  Suppose that $r(m)$ is large and that $r(1)=1$, 
and define $\chr(m):=r(m+1)$, $m\ge0$.
Conditioned on $X_i=x_1$, the results in \refT{TI1} and
\refSs{Slarge}--\ref{Slarge++} hold,
mutatis mutandis, with 
$r(m)$ replaced by $\chr(m)$ and
$n$ replaced by $pn$, where
$p:=\P(X_2>x_1)=e^{-x_1}$.

In particular,  when \refT{TI1} applies, \eqref{limW} extends to
  \begin{align}\label{tcheck}
M_n/n^\ga\asto  W= p^\ga \chW = e^{-\ga X_1} \chW,
  \end{align}
where $\chW$ is independent of $X_1$ and has moments as in \eqref{ti1} for
the sequence $\chr(m)$.
\end{theorem}
For a different distribution of the values $X_n$, \eg{} uniform, $p$ is of
course given by the corresponding tail probability.

\begin{proof}
  This was explained already in \refE{E+1}, although we here modify in the
  opposite direction, so the original sequence here is the modified one there.
As explained in \refE{E+1}, of the first $n$ candidates, the ones accepted
are the first one and then the candidates accepted using the strategy
given by $\chr(m)$ on the candidates that pass the test $X_k>x_1$.
For asymptotics, we can ignore the first accepted candidate, and thus the
results are the same as for $\chr(m)$ with $n$ replaced by
$\chN_n$, the (random) 
number of values $X_k$, $2\le k\le n$, such that $X_k>x_1$.
%$\ch N_n\xsim\Bi\bigpar{n-1,  e^{-X_1}}$ candidates 
By the law of large numbers, \as{} $\chN_n\sim pn$, and the result follows.
We omit the details.
\end{proof}

\begin{example}
  Consider again \hmedian{} as in \refE{Emedian}, but condition on $X_1$.
The sequence $\chr(m):=r(m+1)$ then is the one studied in \refE{Ekapad};
thus we find, for example, see \eqref{tcheck}, that conditioned on $X_1=x_1$,
\begin{align}
M_n/n\qq\asto p\qq \chW  
\end{align}
where $p=e^{-x_1}$ and $W$ has the distribution with density \eqref{Wkapad}.
This (and more) has been shown by \citet[Theorem 5]{HelmiP-median2013}.
Moment convergence holds too, and thus, using \eqref{eleonora}, for
$-2<s<\infty$, 
\begin{align}
  \E \bigpar{M_n^s\mid X_1=x_1} \sim e^{-sx_1/2}\E \chW^s n^{s/2}
=e^{-sx_1/2} 2^s \gG\Bigpar{\frac{s}2+2}  n^{s/2}.
\end{align}
\end{example}

\begin{problem}
  What happens if we condition on $X_1$ in a case with $r(1)=2$?
\end{problem}

\section{Probability of accepting and length of gaps}\label{Sgaps}

Let $I_n:=\indic{X_n \text{ is accepted}}$ and $p_n:=\E I_n$
be the indicator and the probability that candidate
$n$ is accepted; thus $M_n=\sumkn I_k$ and $\E M_n=\sumkn p_k$.
Let $\YYn$ be the current threshold when candidate $n$ is examined.
We then have accepted $M_{n-1}$ candidates, and thus
\begin{align}\label{yyn}
  \YYn&=Y_{M_{n-1}}.
  \end{align}
Furthermore, if $P_n$ is the conditional probability that $X_n$ is accepted
given the past,
  \begin{align}\label{emmPn}
P_n:&=  \E \bigpar{I_n\mid X_1,\dots,X_{n-1}}= e^{-\YYn}=e^{-Y_{M_{n-1}}}
  \end{align}
and thus
  \begin{align}\label{emmpn}
p_n &= \E P_n = \E e^{-Y_{M_{n-1}}}.
\end{align}
We return to a more explicit asymptotic result in the case \eqref{sofie}
in \refT{TPn} below.

Conditioned on $\YYn$, or equivalently on $P_n$,
the waiting time until the next candidate is accepted
is $\Ge(P_n)$. We will see that asymptotically, the same holds if we go back
in time from $n$ to the last acceptance.
The next lemma excludes some extreme cases.

\begin{lemma}\label{LE}
If
\begin{align}\label{le}
  \sumk \frac{\gd_k}{r(k)} = \infty,
\end{align}
then $Y_m\asto\infty$ as \mtoo, and thus $P_n\asto0$ and $p_n\to0$ as \ntoo.

Conversely, if the sum in \eqref{le}
converges, then $Y_m\asto Y_\infty<\infty$.
\end{lemma}
\begin{proof}
It follows from \eqref{eq:ly} that 
\begin{align}\label{Le}
 Y_m\to Y_\infty:=\summ \frac{\gd_k}{r(k)}E_k\le\infty.
\end{align}
Note that the sum in \eqref{le} is $\E Y_\infty$. Hence, if the sum is
finite, then $Y_\infty<\infty$ a.s.

Conversely, assume that \eqref{le} holds. Then the a.s.\
divergence of the sum in \eqref{Le}
follows by the Kolmogorov three series theorem
\cite[Theorem 6.5.5]{Gut}, or by \refL{LZ} in the case of large $r(m)$ and 
otherwise by \refL{LD}, which implies first $y_m/\gs_m\to\infty$ and then
$Y_m/y_m\pto1$. Furthermore, $M_{n-1}\asto\infty$, and thus
$Y_{M_{n-1}}\asto\infty$ as \ntoo; hence $P_n\asto0$ by \eqref{emmPn} and
$p_n\to0$ by \eqref{emmpn} and dominated convergence.
\end{proof}

\begin{theorem}\label{TE}
  Suppose that \eqref{le} holds and that $r(m)\to\infty$ as \mtoo.
Then there exists a sequence $a_n\to\infty$, such that on the interval 
$\cJ_n:=[n-a_n P_n\qw,n]$, the stochastic process $(I_k)_{k\in\cJ_n}$
\whp{} agrees with a sequence $(I'_k)_{k\in\cJ_n}$ of indicator variables that
conditioned on $P_n$ are \iid{} with $I'_k\xsim\Be(P_n)$.
\end{theorem}

\begin{proof}
%, let $n_1:=\ceil{n-aP_n\qw}$ and  $\cJ_n:=[n_1,n]$. 
Fix an integer $K>0$, and define the stopping time
$\xtau_n:=\min \set{k: M_{k-1}\ge M_{n-1}-K}$.
Thus $\xtau_n\le n$ and, assuming $n$ is so large that $M_{n-1}\ge K$, 
\begin{align}
\YYx{\xtau_n}=
  Y_{M_{n-1}-K}\ge Y_{M_{n-1}} - \frac{K}{r(M_{n-1}-K)} = Y_{M_{n-1}}+o(1)
= \YYn+o(1)
\end{align}
\as{} as \ntoo, since $M_n\to\infty$ and thus $r(M_{n-1}-K)\to\infty $ a.s.
Consequently, \as{} as \ntoo, 
\begin{align}\label{PoP}
  P_{\xtau_n} \sim P_n.
\end{align}

By definition, $K$ candidates are accepted in the interval $\cJx:=[\xtau_n,n)$
(provided $\xtau_n>1$), and, conditioned on $P_{\xtau_n}$,
each candidate in $\cJx$ is accepted with probability at most $P_{\xtau_n}$.

Let $a$ be a fixed large number and define $n_1:=\ceil{n- a P_{\xtau_n}\qw}$.
If $\xtau_n\ge n_1$, then 
$|\cJx|\le n-n_1\le a P_{\xtau_n}\qw$, and thus at least $K$
candidates are accepted in the interval $[\xtau_n,\xtau_n+\floor{a P_{\xtau_n}\qw})$.
Hence, using Markov's inequality,
\begin{align}
  \P\bigpar{\xtau_n\ge n_1 \mid P_{\xtau_n}=p}
\le \P\bigpar{\Bin\bigpar{\floor{a p\qw}, p}\ge K}
\le a/K.
\end{align}
Consequently, given any $\eps>0$, we may by choosing $K>a/\eps$ make this
probability $<\eps$, uniformly in $p>0$. 
Hence, we may in the rest of the proof assume that
$\xtau_n<n_1$. This means that
for every $k$ in the
interval $[n_1,n]$,
$\xtau_n<k\le n$, and thus $P_{\xtau_n} \ge P_k\ge P_n$.
It follows that, conditioned on $P_{\xtau_n}$, we may couple the Markov process
$(I_k)_{k\in[n_1,n]}$ with a sequence of (conditionally) \iid{} variables 
$(I''_k)_{k\in[n_1,n]}$ with $\P(I''_k)=P_{\xtau_n}$, with an error probability
at most, using \eqref{PoP},
\begin{align}\label{winston}
  (n-n_1+1) \bigpar{P_{\xtau_n}-P_n} \sim a P_{\xtau_n}\qw\bigpar{P_{\xtau_n}-P_n}
\asto0.
\end{align}
We now uncondition, and see (using \eqref{winston} and dominated
convergence)
that we may couple 
$(I_k)_{k\in[n_1,n]}$
and
$(I''_k)_{k\in[n_1,n]}$ with error probability $o(1)$.
We may then instead couple with 
$(I_k')_{k\in[n_1,n]}$ where $I_k'$ are conditionally \iid{} with 
$\P(I'_k)=P_{n}$, introducing an additional error $o(1)$ by 
an estimate similar to \eqref{winston}.

Furthermore, \as{} $P_{\xtau_n}\sim P_n$ by \eqref{PoP}, and thus 
$aP_{\xtau_n}\qw > (a-1)P_n\qw$ for large $n$.
Thus, we may as well couple $(I_k)$ and $(I_k')$ on $[n-(a-1)P_n\qw,n]$,
for any fixed $a$.

We may here replace $a$ by $a+1$. Moreover, by a simple general argument,
since this coupling with error probability $o(1)$ is possible for every
fixed $a>0$, it is also possible for some sequence $a_n\to\infty$; this
follows by the following elementary lemma, taking $x(a,n)$ to be the total
variation distance between the two sequences, which completes the proof.
\end{proof}

\begin{lemma}
  Suppose that $x(a,n)$, $a,n\in \bbN$, 
are real numbers such that for every fixed $a$,
  $x(a,n)\to0$ as \ntoo. Then there exists a sequence $a_n\to\infty$ such
  that $x(a_n,n)\to0$.
\end{lemma}

\begin{proof}
  Let $n_0=1$. For every $k\ge1$, choose $n_k>n_{k-1}$ 
such that $|x(k,n)|<1/k$ when $n\ge n_k$. Define $a_n=k$ when $n_k \le
n<n_{k+1}$. 
\end{proof}

Let $L_n:=n-N_{M_n}$ be the number of candidates examined after the last
accepted one.
Let $\dtv(X,Y)$ denote the total variation distance between two
distributions or random variables.

\begin{corollary}\label{CE}
    Suppose that \eqref{le} holds and that $r(m)\to\infty$ as \mtoo.
Then, conditioned on $P_n$, $\dtv\bigpar{L_n,\Ge(P_n)}\to0$ as \ntoo.
Consequently, still conditioned on $P_n$,
\begin{equation}
  P_n L_n \dto \Exp(1).
\end{equation}
\end{corollary}
\begin{proof}
  Immediate from \refT{TE} and the fact that $P_n\asto0$ by Lemma \ref{LE}.
\end{proof}

We can also find the unconditional distribution of $L_n$.
For convenience, and in order to obtain more explicit results, we consider
only the case in \refS{Slinear}, and we assume $\ga<1$, which implies
\eqref{le}.
We first study $P_n$.

\begin{theorem}\label{TPn}
Suppose that \eqref{sofie} holds for some $\ga\in(0,1)$. 
Then, a.s.,
\begin{equation}
  \label{tpn}
  P_n\sim \ga M_n/n \sim \ga W n^{\ga-1},
\end{equation}
where $W$ is as in \refT{TI1}.
\end{theorem}

\begin{proof}
  Let, for convenience $\xi:=\bigpar{\ga\qw-1}\gam+\rho+1$, so 
$W=e^{-\ga\xi-\ga Z}$ by \eqref{W}.
Then, by \refL{LZ} and \eqref{la}, a.s.,
\begin{equation}
  Y_m=y_m+Z+o(1)
= \bigpar{\ga\qw-1}\log m -\log\ga + \xi + Z + o(1)
\end{equation}
and thus, by \eqref{W},
\begin{align}
  \YYn 
&=\bigpar{\ga\qw-1}\log M_n -\log\ga +\xi+Z+o(1)
\notag\\&
=(1-\ga)\log n +\bigpar{\ga\qw-1}\log W -\log\ga +\xi+Z+o(1)
\notag\\&
=(1-\ga)\log n -\log\ga +\ga\xi+\ga Z+o(1). \label{yyn2}
\end{align}
Hence, by \eqref{emmPn},
\begin{align}
  P_n\sim \ga e^{-\ga\xi-\ga Z} n^{\ga-1}
= \ga W n^{\ga-1}.
\end{align}
\end{proof}

\begin{theorem}
Suppose that \eqref{sofie} holds for some $\ga\in(0,1)$. 
Then
\begin{align}
  L_n/n^{1-\ga} \dto \LL:=\ga\qw W\qw E,
\end{align}
where $W$ is as in \refT{TI1} and $E\xsim\Exp(1)$ is independent of $W$.
\end{theorem}

\begin{proof}
  A consequence of \refC{CE} and \refT{TPn}.
\end{proof}

\begin{remark}
  The moments 
  \begin{equation}\label{kd}
    \E \LL^s = \ga^{-s}\E \bigsqpar{W^{-s}} \E \bigsqpar{E^s}
= {\ga^{-s}}\gG(s+1)\E W^{-s}, 
\qquad s>-1,
  \end{equation}
follow from \eqref{ti1}. Note that 
for real $s$, $\E\LL^s<\infty$ if $-1<s<\rx/\ga$, but
not outside this interval, see  \refR{RZ1}.
\end{remark}

\begin{example}
  For \hmedian, $W$ has the density function \eqref{fWmed}, and thus 
$\LL=2 W\qw E$ has the density function
\begin{align}
  f_{\LL}(x)&=\intoo f_W(y) f_{2y\qw E}(x) \dd y 
=\intoo \frac{y}2 e^{-y^2/4} 
\frac{y}2  e^{-yx/2} \dd y
\nonumber\\&
=2\intoo t^2 e^{-t^2-tx} \dd t,
\end{align}
as found by \cite[Theorem 3]{HelmiP-median2013}.
Furthermore, \eqref{kd} and \eqref{EWmed} yield
\begin{align}
  \E \LL^s =\gG(s+1) \gG\Bigpar{1-\frac{s}{2}},
\qquad -1<s\le 2.
\end{align}
In particular,  $\E\LL=\sqrt\pi$. 
\cite{HelmiP-median2013} proved convergence of $\E L_n/n\qq$ to this limit.
\end{example}

\begin{remark}
  \refT{TE} implies also the same limit results for, e.g., the distance
  between the last two accepted.
See \cite[Theorem 4]{HelmiP-median2013} for \hmedian.
\end{remark}

\section{The distribution of accepted values}\label{Saccepted}

Finally, we study the distribution of the accepted values. For simplicity we
consider again only the situation in \refS{Slinear}. We also assume for
simplicity that $\ga<1$, leaving the case $\ga=1$ to the reader.

Let, for a real number $x$,
$\Mle{x}_n$ 
be the number of values $X_k$ with $k\le n$
that are accepted and furthermore satisfy $X_k\le x$.
Define $\Mgt{x}_n=M_n-\Mle{x}_n$ similarly.

\begin{theorem}
  \label{Tle}
Suppose that \eqref{sofie} holds for some $\ga\in(0,1)$. 
Then, a.s., for every $u\in \bbR$,  
\begin{align}\label{tle}
  \frac{\Mle{\YYn+u}_n}{M_n}\to F(u):=
  \begin{cases}
(1-\ga) e^{\ga u/(1-\ga)},  & u\le 0,
\\ 
1-\ga e^{-u}, & u\ge0.
  \end{cases}
\end{align}
In other words, the empirical distribution function of the differences
$X_k-\YYn$ 
for the $M_n$ accepted candidates converges \as{} to the distribution with
distribution function $F(u)$.
Hence, if $\bX_n$ is the value of one of the $M_n$ accepted
candidates, chosen uniformly at random, then, 
%also conditioned on the accepted values, a.s.,
\begin{align}
  \bX_n-\YYn \dto V,
\end{align}
where $V$ has the distribution $F(u)$.
\end{theorem}

The proof is given later.
Note that $V$ has density
\begin{align}
  f(u):=F'(u)=
  \begin{cases}
\ga e^{\ga u/(1-\ga)},  & u<0,
\\ 
\ga e^{-u}, & u>0.
  \end{cases}
\end{align}
Thus, $V$ has an asymmetric double exponential distribution (Laplace
distribution); if $\ga=1/2$, $V$ has the usual Laplace distribution.

In order to prove \refT{Tle}, we introduce a simpler strategy.
Fix %(besides $\ga$)
a real number $z$ and define
\begin{align}\label{xn}
  x_n=x_n(z):=(1-\ga)\log n + \ga z -\log \ga
\end{align}
and, for later convenience,
\begin{align}\label{wz}
  w=w(z):=e^{-\ga z}\in\ooo.
\end{align}
Define the strategy $\Hz$ as 'accept if $X_n>x_n$'. (Thus $\Hz$ is
not a rank-based strategy.)

\begin{lemma}
  \label{LHz}
For the strategy $\Hz$, a.s., 
\begin{align}
  M_n\sim w n^\ga \label{lhz1}
\end{align}
and, for every real $u$,
\begin{align}\label{lhz2}
  \frac{\Mle{x_n+u}_n}{M_n}\to F(u):=
  \begin{cases}
(1-\ga) e^{\ga u/(1-\ga)},  & u\le 0,
\\ 
1-\ga e^{-u}, & u\ge0.
  \end{cases}
\end{align}
\end{lemma}
\begin{proof}
Consider first a fixed $u\ge0$.
Then every value $X_k>x_n+u$ with $k\le n$ will have $X_k>x_n\ge
x_k$, and thus be accepted. Hence, $\Mgt{x_n+u}_n$ is the number of all such
values, and since the indicators $\indic{X_k>x_n+u}$ are \iid{} with
$\P(X_k>x_n+u)=e^{-x_n-u}$,
\begin{align}
  \Mgt{x_n+u}_n\xsim\Bi\bigpar{n, e^{-x_n-u}}.
\end{align}
This binomial random variable has mean $ne^{-x_n-u}=\ga w e^{-u}n^\ga$,
recalling \eqref{xn} and \eqref{wz}.
A standard Chernoff bound, see \eg{} \cite[Corollary 2.3]{JLR}, shows that
for every $\eps>0$,
\begin{align}\label{chernoff}
  \P\bigpar{\bigabs{\Mgt{x_n+u}_n/\E \Mgt{x_n+u}_n-1}>\eps} 
  \le 2 \exp\bigpar{-c(\eps,u)n^\ga}.
\end{align}
It follows, by the Borel--Cantelli lemma, that a.s.
$\Mgt{x_n+u}_n/\E \Mgt{x_n+u}_n\to1$, \ie,
\begin{align}\label{emma}
\Mgt{x_n+u}_n \sim \E \Mgt{x_n+u}_n 
%= ne^{-x_n-u}
=\ga wn^{\ga}e^{-u}.
\end{align}

Consider next $M_n$. This too is a sum of independent indicators
$I_k:=\indic{X_k>x_k}$. 
Furthermore,  for $k$ large enough so that $x_k\ge0$,
\begin{align}
  p_k:=\E I_k = e^{-x_k}=\ga e^{-\ga z}k^{\ga-1}
=\ga w k^{\ga-1}.
\end{align}
Since $M_n=\sumkn I_k$, we have
\begin{align}\label{jesper}
  \E M_n=\sumkn p_k = \sumkn \ga w k^{\ga-1} + O(1)
=wn^\ga+O(1).
\end{align}
The random variables $I_k$ are not identically distributed, but the Chernoff
bound holds for sums of arbitrary independent indicator variables
\cite[Theorem 2.8]{JLR}, and thus \eqref{chernoff} holds for $M_n$ too, and
we obtain as above
\begin{align}
  M_n/\E M_n\asto 1,
\end{align}
which together with \eqref{jesper} yields \eqref{lhz1}.

Furthermore, \eqref{lhz1} and \eqref{emma} show that for every fixed $u\ge0$,
\begin{align}
\xfrac{\Mgt{x_n+u}_n}{M_n}
\asto \ga e^{-u},
\end{align}
and thus \eqref{lhz2} holds for $u\ge0$.

Finally, consider a fixed $u\le 0$. Similarly as above, we write 
$\Mle{x_n+u}_n$ as a sum of independent indicators 
$I_k':=\indic{x_k < X_k\le  x_n+u}$. Note that $I'_k= 0$ unless
$x_k<x_n+u$, which by \eqref{xn} is equivalent to
\begin{align}
  (1-\ga)\log k <(1-\ga)\log n + u
\end{align}
or
\begin{align}
k< e^{u/(1-\ga)}n.   
\end{align}
For such $k$, except possibly for some small $k$ with $x_k<0$,
\begin{align}
  \E I_k':=e^{-x_k}-e^{-(x_n+u)}
=\ga w k^{\ga-1} - \ga w e^{-u}n^{\ga-1}.
\end{align}
Hence, 
\begin{align}\label{Emle}
  \E \Mle{x_n+u}_n
&= \sum_{k=1}^{e^{u/(1-\ga)}n} 
\bigpar{\ga w k^{\ga-1} - \ga w e^{-u}n^{\ga-1}} +O(1)
\notag\\&
=w \bigpar{e^{ u/(1-\ga)}n}^\ga -\ga w e^{-u} e^{u/(1-\ga)} n^\ga +O(1)
\notag\\&
=wn^\ga e^{\ga u/(1-\ga)}\xpar{1 -\ga} +O(1).
\end{align}
The Chernoff argument applies again, and yields
\begin{align}
     \Mle{x_n+u}_n/  \E \Mle{x_n+u}_n\asto1
\end{align}
and thus, using \eqref{Emle} and \eqref{lhz1}, a.s.,
\begin{align}
     \Mle{x_n+u}_n \sim (1-\ga)wn^\ga e^{\ga u/(1-\ga)}
\end{align}
and
\begin{align}
     \Mle{x_n+u}_n /M_n \to (1-\ga) e^{\ga u/(1-\ga)} = F(u).
\end{align}
We have proved that \eqref{lhz2} holds \as{} for every fixed
$u\in\bbR$. Hence, it holds \as{} for every rational $u$, but this implies
that it holds for all $u$ simultaneously, 
since the \lhs{} is monotone in $u$ and the
\rhs{} is continuous; we omit the details.
\end{proof}

\begin{proof}[Proof of \refT{Tle}]
Consider the rank-based strategy (as in the rest of the paper), which we
denote by $\cR$, together
with the strategies $\cH(z)$ for all rational $z\in\bbR$, acting on the same
sequence $X_n$. We indicate quantities for the strategy $\cH(z)$ with an
extra argument $z$; for example, $M_{n}(z)$ for the number of accepted
values of the $n$ first ones.

Recall that the strategy $\cR$ is to accept $X_n$ if $X_n>\YYn$. 
By \eqref{yyn2} and \eqref{xn}, if $z<Z+\xi$, then \as{} $\YYn>x_n(z)$ for
all large $n$, and thus (for large $n$) every value accepted by $\cR$ is
also accepted by $\cH(z)$. Consequently, \as{} 
\begin{equation}\label{mp}
M_n \le M_n(z) +O(1). 
\end{equation}
Similarly, for every fixed real $u$, \as,
\begin{equation}\label{mpu}
\Mgt{\YYn+u}_n \le \Mgt{x_n(z)+u}_n(z)+O(1).
\end{equation}
For every rational $z>Z+\xi$, \eqref{mp} and \eqref{mpu} hold \as{} 
with the inequalities in the opposite direction.

Consequently, 
with $G(u):=1-F(u)$,
\as, for every rational $z<Z+\xi$ and $z'>Z+\xi$, 
using \eqref{lhz1}, \eqref{wz}, and \eqref{lhz2},
\begin{align}
    \frac{\Mgt{\YYn+u}_n}{M_n}
&\le  \frac{\Mgt{x_n(z)+u}_n(z)+O(1)}{M_n(z')+O(1)}
=  \frac{\Mgt{x_n(z)+u}_n(z)+O(1)}{M_n(z)}
 \frac{M_n(z)}{M_n(z')+O(1)} 
\notag\\&\to G(u)\frac{w(z)}{w(z')}
= G(u) e^{\ga(z'-z)}.
\end{align}
Hence, \as,
\begin{equation}
\limsup_\ntoo \frac{\Mgt{\YYn+u}_n}{M_n}
\le G(u) e^{\ga(z'-z)}
\end{equation}
for every rational $z$ and $z'$ with $z<Z+\xi<z'$. Consequently, \as
\begin{equation}
\limsup_\ntoo \frac{\Mgt{\YYn+u}_n}{M_n}
\le G(u)=1-F(u).
\end{equation}
A lower bound follows in the same way, now comparing in the opposite
directions with $z>Z+\xi>z'$.
This proves \eqref{tle} \as{} for a fixed $u$, and thus for all rational
$u$ simultaneously, which again implies the  result for all real $u$
simultaneously by monotonicity and continuity.
\end{proof}

\begin{corollary}
  Suppose that \eqref{sofie} holds for some $\ga\in(0,1)$.
Then, 
the fraction of the accepted values that are larger than the current
threshold, and thus would have been accepted now, converges \as{} to $\ga$.
\end{corollary}
\begin{proof}
  This fraction is $\Mgt{\YYn}_n/M_n$, so the result is the case $u=0$ of
  \refT{Tle}. 
\end{proof}

\appendix
\section{Proof of \eqref{sofie2}}\label{A1}

\begin{lemma}
Suppose that \eqref{sofie} holds for some $\ga\in(0,1]$. Then \eqref{sofie2}
holds. 
\end{lemma}
\begin{proof}
  Fix $\gd>0$ and suppose that $m$ is such that $r(m)\ge (1+\gd)^{2} \ga m$.
Then, for every $k$ with $m\le k\le(1+\gd)m$, we have $r(k)\ge r(m)$ and thus
\begin{align}
  (\ga k)\qw - r(k)\qw \ge (1+\gd)\qw (\ga m)\qw -(1+\gd)\qww (\ga m)\qw
=c(\gd) m\qw.
\end{align}
Hence,
\begin{align}\label{johan}
  \sum_{k=m}^\infty\bigabs{r(k)\qw-(\ga k)\qw}
\ge \sum_{k=m}^{\floor{(1+\gd)m}} c(\gd) m\qw \ge c(\gd).
\end{align}
Since the sum in \eqref{sofie} converges, \eqref{johan} cannot hold for
arbitrarily large $m$, and thus $r(m)< (1+\gd)^{2} \ga m$ for large $m$.
Similarly, if $0<\gd<1$, then
$r(m)> (1-\gd)^{2} \ga m$ for large $m$.
This proves \eqref{sofie2}.
\end{proof}

\newcommand\AAP{\emph{Adv. Appl. Probab.} }
\newcommand\JAP{\emph{J. Appl. Probab.} }
\newcommand\JAMS{\emph{J. \AMS} }
\newcommand\MAMS{\emph{Memoirs \AMS} }
\newcommand\PAMS{\emph{Proc. \AMS} }
\newcommand\TAMS{\emph{Trans. \AMS} }
\newcommand\AnnMS{\emph{Ann. Math. Statist.} }
\newcommand\AnnPr{\emph{Ann. Probab.} }
\newcommand\CPC{\emph{Combin. Probab. Comput.} }
\newcommand\JMAA{\emph{J. Math. Anal. Appl.} }
\newcommand\RSA{\emph{Random Struct. Alg.} }
\newcommand\ZW{\emph{Z. Wahrsch. Verw. Gebiete} }
\newcommand\DMTCS{\jour{Discr. Math. Theor. Comput. Sci.} }

\newcommand\AMS{Amer. Math. Soc.}
\newcommand\Springer{Springer-Verlag}
\newcommand\Wiley{Wiley}

\newcommand\vol{\textbf}
\newcommand\jour{\emph}
\newcommand\book{\emph}
\newcommand\inbook{\emph}
\def\no#1#2,{\unskip#2, no. #1,} %(typeset after year) 
\newcommand\toappear{\unskip, to appear}

\newcommand\arxiv[1]{\texttt{arXiv:#1}}
\newcommand\arXiv{\arxiv}

\def\nobibitem#1\par{}

%\newpage

\end{document}